\documentclass[12pt]{amsart}

\usepackage[T1]{fontenc}
\usepackage{fullpage}
\usepackage{lmodern,microtype}
\usepackage{amsmath,amssymb,amsthm,mathtools,mathrsfs}
\usepackage{graphicx}
\usepackage{enumitem,needspace}
\usepackage{array,booktabs}
\usepackage{xcolor}
\usepackage{tikz,tikz-cd}
\usetikzlibrary{positioning}
\usepackage[english]{babel}
\usepackage{hyperref}
\usepackage[capitalise,nameinlink,noabbrev]{cleveref}

\hypersetup{
  colorlinks=true,
  linkcolor=blue,
  citecolor=blue,
  filecolor=magenta,
  urlcolor=cyan,
  pdftitle={Uniform local models of quasi-unipotent p-adic differential modules},
  pdfsubject={Uniform analytic realization of differential and Frobenius modules},
  pdfauthor={Poornima Belvotagi},
  pdfkeywords={p-adic differential equations, Robba rings, quasi-unipotence, local monodromy, Frobenius structures}
}
\setlist[enumerate]{leftmargin=*,itemsep=4pt}
\setlist[itemize]{leftmargin=*,itemsep=3pt}

\theoremstyle{plain}
\newtheorem{maintheorem}{Theorem}

\newtheorem{theorem}{Theorem}[subsection]
\newtheorem{lemma}[theorem]{Lemma}
\newtheorem{proposition}[theorem]{Proposition}
\newtheorem{corollary}[theorem]{Corollary}

\theoremstyle{definition}
\newtheorem{definition}[theorem]{Definition}
\newtheorem{remark}[theorem]{Remark}
\numberwithin{equation}{section}
\crefname{maintheorem}{Theorem}{Theorems}
\Crefname{maintheorem}{Theorem}{Theorems}

\DeclareMathOperator{\GL}{GL}
\DeclareMathOperator{\rk}{rank}
\DeclareMathOperator{\Aut}{Aut}
\DeclareMathOperator{\End}{End}
\DeclareMathOperator{\Irr}{Irr}
\DeclareMathOperator{\Ind}{Ind}
\DeclareMathOperator{\Sw}{Sw}
\DeclareMathOperator{\QU}{QU}
\newcommand{\del}{\delta}
\newcommand{\OO}{\mathcal{O}}
\newcommand{\LL}{\mathcal{L}}
\newcommand{\FM}{\mathsf{F}}
\newcommand{\Q}{\mathbb{Q}}
\newcommand{\ee}{\boldsymbol{e}}
\newcommand{\ff}{\boldsymbol{f}}
\newcommand{\norm}[1]{\left\lVert#1\right\rVert}

\title[Uniform local models]{Uniform local models of\texorpdfstring{\protect\\}{ }quasi-unipotent $p$-adic differential modules}
\author{Poornima Belvotagi}
\address{Department of Mathematics, UC San Diego, La Jolla, CA 92093-0112, USA}
\email{pb@ucsd.edu}
\thanks{The author was supported in part by National Science Foundation grant DMS-2401536. Additional support was provided through National Science Foundation grant DMS-2053473 and the S.~E.~Warschawski Professorship, both held by Kiran Kedlaya.}
\date{October 2026}

\begin{document}
\begin{abstract}
We construct uniform analytic models of quasi-unipotent differential modules over Robba rings. For modules of rank at most $r$ and highest ramification break at most $b$, the inner radius of the model annulus depends only on $r,b$ and the residue characteristic $p$. At each radius, the model admits a basis, possibly depending on the radius, in which the connection matrix is bounded by a constant depending only on $r,b,p$. Both bounds are independent of the coefficient field, and these analytic models require no coefficient extension. A prescribed absolute Frobenius structure extends to such a bounded model with the same radius and connection bound. After a finite coefficient extension, we also realize quasi-constant modules, which carry the finite local monodromy, as differential direct summands of sums of pushforwards of rank-one modules, with uniformly bounded cover degrees and highest breaks of the inducing modules.
\end{abstract}
\maketitle
\enlargethispage{-12pt}

\setcounter{tocdepth}{1}
\tableofcontents

\section{Introduction}\label{sec:intro}
Overconvergent $F$-isocrystals play the role in $p$-adic cohomology of varieties over a finite field that lisse $\ell$-adic sheaves play in \'etale cohomology. On a smooth curve, the behavior of such an object near a missing point is recorded by its local monodromy: a differential module with Frobenius structure over a Robba ring. The $p$-adic local monodromy theorem, conjectured by Crew \cite{Cre98} and proved independently by Andr\'e \cite{And02}, Mebkhout \cite{Meb02}, and Kedlaya \cite{Ked04}, asserts that every such module is quasi-unipotent, in close analogy with Grothendieck's $\ell$-adic local monodromy theorem \cite[Appendix]{ST68}. Matsuda's decomposition \cite{Mat02}, in the form of \cite[Theorem~4.6.6]{Ked05}, writes a quasi-unipotent module as a direct sum of tensor products of irreducible quasi-constant modules and unipotent modules. The quasi-constant factors carry the finite local monodromy, while the unipotent factors have constant nilpotent connection matrices. We recall the precise statement in \cref{thm:matsuda}.

Every differential module over the Robba ring already descends to some annulus near the outer boundary. The question here is whether one annulus works for all objects with bounded rank and ramification. Qualitative descent gives no bound for the inner radius in terms of these data. Even the degree of a trivializing cover need not be bounded by these data: a tame rank-one module with exponent $1/n$, where $(n,p)=1$, has highest break zero and connection norm one, but trivializing it requires ramification index divisible by $n$; compare \cite[Remark~6.1.7]{Ked05}. We instead control the degrees of inducing covers while allowing the inducing rank-one modules to retain nontrivial monodromy. Auxiliary coefficient extensions are used in the construction, but their degrees do not enter the final analytic bounds.

\Cref{thm:A} supplies such an annulus, together with a radiuswise connection bound, over the original coefficient field. Its constants depend only on the rank, the highest ramification break, and the residue characteristic. For a module with a prescribed absolute Frobenius structure, \cref{thm:B} shows that a bounded model with the same parameters can be chosen to carry it. The witnessing bases may depend on the radius; see \cref{rem:globalbasis}. \Cref{subsec:constants} explains how the constants are obtained.

The structural input is a bounded rank-one realization, stated separately as \cref{thm:C}. After finite coefficient extension, quasi-constant modules are differential direct summands of sums of rank-one pushforwards. Both the total cover degree and the breaks of the inducing rank-one modules are controlled by the original rank and break. This packages the finite monodromy data into bounded-degree covers before the analytic estimates are applied.

A motivation comes from cut-by-curves criteria for overconvergence. Grubb--Kedlaya--Upton prove a criterion under a geometric tameness hypothesis and ask whether a bound on curvewise Swan conductors in terms of boundary intersections suffices \cite[Theorem~9.7 and Question~10.2]{GKU22}. For a local module of rank at most $r$, one has $b(M)\le\Sw(M)\le r\,b(M)$; the comparison with the differential Swan conductor is recalled in \cref{subsec:setup}. On a family of test curves with bounded boundary intersections, their proposed conductor bound therefore gives a common bound for the highest breaks. The local input supplied by \cref{thm:B} is a common annulus carrying the prescribed Frobenius and a uniform radiuswise connection bound, independently of the coefficient fields. Comparing these models across curves requires additional analytic compatibility and is left to future work.

For a module admitting Frobenius, \cite[Lemma~6.1.3]{Ked05} gives arbitrarily small connection matrices in suitable bases on closed subannuli near the outer boundary, but gives no bound for the starting radius in terms of $r,b,p$. Pulita's explicit rank-one classification and Frobenius constructions \cite{Pul07}, building on Matsuda's work \cite{Mat95}, provide the rank-one input used here. The present construction combines these models with bounded-degree induction, uniform pushforward estimates, and coefficient descent. The inducing-index bound avoids the wild-inertia order estimate of \cite[Proposition~2.3.9]{Ked05}, which fails in general; see \cref{rem:earliertameindex}.

\subsection{Setup and notation}\label{subsec:setup}

Let $K$ be a complete discretely valued field of characteristic zero with perfect residue field $k$ of characteristic $p>0$. We normalize its absolute value by $|p|=p^{-1}$. Write $\OO_K$ for its valuation ring and $\mathfrak m_K$ for the maximal ideal. The canonical Witt coefficient field
\[
 C_k=\operatorname{Frac}W(k)\subseteq K
\]
has $[K:C_k]=e(K)<\infty$, where $e(K)$ is the absolute ramification index of $K$. Annuli of arbitrary real radii are understood in the generalized affinoid sense, equivalently through their convergent Laurent-series algebras; see \cite[Remark~3.2.4]{Ked05}. For $0<s<1$, let
\[
 A_{K,t}[s,1)=\{s\le |t|<1\},\qquad
 \mathscr A_s=\Gamma(A_{K,t}[s,1),\OO),\qquad
 \mathcal{R}_{K,t}=\bigcup_{0<s<1}\mathscr A_s.
\]
The ring $\mathcal{R}_{K,t}$ is the Robba ring. When the coefficient field and coordinate are fixed, we abbreviate $A_{K,t}[s,1)$ to $A[s,1)$.

The bounded Robba ring $\mathcal{R}^{\mathrm{bd}}_{K,t}$ consists of Laurent series in $\mathcal{R}_{K,t}$ with bounded coefficients. Its valuation ring
\[
 \mathcal{R}^{\mathrm{int}}_{K,t}
 =\left\{\sum_{n\in\mathbb{Z}}a_nt^n\in\mathcal{R}_{K,t}:a_n\in\OO_K\right\}
\]
is henselian \cite[Proposition~3.5.5]{Ked05} and has residue field $k((\bar t))$. Finite separable extensions of this residue field therefore lift to finite \'etale extensions of $\mathcal{R}^{\mathrm{int}}_{K,t}$ and hence to finite \'etale extensions of $\mathcal{R}_{K,t}$.

Throughout, the covers of Robba rings used below are obtained from these finite \'etale integral lifts; this is what we mean by a \emph{formally \'etale cover}. On a sufficiently small outer annulus they are finite \'etale analytic maps. Pushforward means restriction of scalars with its induced connection.

A differential module is \emph{quasi-constant} if it becomes constant on such an extension. It is \emph{unipotent} if it becomes constant after adjoining $\log t$, and \emph{quasi-unipotent} if it becomes unipotent after a finite \'etale extension of this type.

\paragraph{Breaks.}
The local monodromy of a quasi-constant module gives a finite-image representation of inertia; we recall its construction in \cref{subsec:monodromy}.

For a quasi-unipotent module $M$, define $b(M)$ to be the maximum of the highest upper-numbering breaks of the inertia representations of its quasi-constant constituents in Matsuda's decomposition; see \cref{thm:matsuda,lem:coeffsplitting} in \cref{subsec:matsuda}. We set this maximum to zero if there is no positive break. For a representation $V$ with finite inertia image, including a character, $b(V)$ denotes its highest upper-numbering break with the same convention. Let $\QU(r,b)$ be the class of quasi-unipotent modules of rank at most $r$ and highest break at most $b$.

Since $k$ is perfect, the differential and monodromy break decompositions agree \cite[Theorem~5.4.10 and Corollary~5.4.11]{Ked05}. Thus the differential Swan conductor $\Sw(M)$ is the sum of the monodromy breaks counted with multiplicity, while $b(M)$ is their maximum. In particular, for a quasi-unipotent module of rank $d\le r$,
\[
 b(M)\le\Sw(M)\le d\,b(M)\le r\,b(M).
\]

\paragraph{Circle norms.}
Put $\del_t=t\,d/dt$. For $\rho\in(0,1)$, set
\[
 R_\rho=\left\{\sum_{n\in\mathbb{Z}}a_nt^n:
        |a_n|\rho^n\longrightarrow0\text{ as }n\longrightarrow\pm\infty\right\},
 \qquad
 \left|\sum_n a_nt^n\right|_\rho=\max_n|a_n|\rho^n.
\]
Thus $R_\rho=\Gamma(A_{K,t}[\rho,\rho],\OO)$, including when $\rho\notin|K^\times|^{\Q}$. For matrices, $|\cdot|_\rho$ denotes the maximum of the norms of the entries.

\paragraph{Models.}
A \emph{differential direct summand} is the image of a horizontal idempotent, or equivalently a differential submodule admitting a differential complement. We use vector bundles for the intrinsic notion of an analytic model and finite free modules when a global basis is chosen. Images of idempotents in these free ambient models are finite projective modules, viewed as vector bundles; no global basis of such an image is assumed.

\begin{definition}\label{def:model}
A \emph{model} of $M$ on $A_{K,t}[s,1)$ is a vector bundle with connection $\mathcal{E}$, together with an isomorphism
\[
 \Gamma(A_{K,t}[s,1),\mathcal{E})\otimes_{\mathscr A_s}\mathcal{R}_{K,t}\simeq M
\]
of differential modules. For $\rho\in[s,1)$, write
\[
 \mathcal{E}_\rho=\Gamma(A_{K,t}[s,1),\mathcal{E})\otimes_{\mathscr A_s}R_\rho
\]
for its restriction to the circle of radius $\rho$. The module $M$ is \emph{$(s,N)$-bounded} if it admits such a model and, for every $\rho\in[s,1)$, an $R_\rho$-basis $\ee_\rho$ of $\mathcal{E}_\rho$ in which the matrix $A_\rho$ of $\nabla_{\del_t}$ satisfies $|A_\rho|_\rho\le N$.
\end{definition}

The bases in this definition may depend on $\rho$.

\begin{definition}
A \emph{witness} for $(s,N)$-boundedness is a model $\mathcal{E}$ and a family $(\ee_\rho)_{s\le\rho<1}$ as in \cref{def:model}. A \emph{global-basis witness} is a model with a single global basis whose restrictions satisfy those bounds. A collection of modules is $(s,N)$-bounded if every member has such a witness.
\end{definition}

Radiuswise boundedness and boundedness with a global-basis witness are monotone: one may increase $s$ and $N$.

\subsection{Main results}
\begin{maintheorem}[Uniform differential models]\label{thm:A}
There are functions
\[
 s=s(r,b,p)\in(0,1),\qquad N=N(r,b,p)\ge1
\]
such that every $M\in\QU(r,b)$ over $\mathcal{R}_{K,t}$ is $(s,N)$-bounded over $K$. The constants are independent of the coefficient field $K$.
\end{maintheorem}

For the second statement, suppose that $K$ has an isometric endomorphism $\sigma_K$ lifting the absolute $p$-power Frobenius of $k$, and use
\[
 \sigma\left(\sum_n a_nt^n\right)=\sum_n\sigma_K(a_n)t^{pn}.
\]
The coefficient lift $\sigma_K$ is automatically an automorphism: isometry sends a uniformizer to a uniformizer, and completeness and surjectivity on the residue field give surjectivity on $\OO_K$ by successive approximation. A Frobenius structure is a horizontal isomorphism $\Phi:\sigma^*M\xrightarrow{\sim}M$. By the $p$-adic local monodromy theorem \cite{Ked04,Ked22}, a module carrying such a structure is quasi-unipotent, so its highest break $b(M)$ is defined. The map $t\mapsto t^p$ sends $A_{K,t}[s^{1/p},1)$ into $A_{K,t}[s,1)$, so a Frobenius structure on a model is defined on this outer subannulus.

\begin{definition}\label{def:frobenius-model}
For $m\ge1$, a \emph{$(\sigma^m,\nabla)$-model} of a Robba module with a prescribed Frobenius structure for $\sigma^m$ is a differential model $\mathcal{E}$ on $A[s,1)$ together with a horizontal isomorphism
\[
 (\sigma^m)^*\mathcal{E}\xrightarrow{\sim}\mathcal{E}|_{A[s^{1/p^m},1)}
\]
inducing the prescribed structure.
\end{definition}

\begin{maintheorem}[Uniform Frobenius models]\label{thm:B}
Let $(M,\Phi)$ be a differential module of rank at most $r$ and highest break at most $b$ over $\mathcal{R}_{K,t}$, with a Frobenius structure for the absolute lift $\sigma$. Put
\[
 s=s(r,b,p),\qquad N=N(r,b,p),
\]
with the functions of \cref{thm:A}. There is an $(s,N)$-bounded $(\sigma,\nabla)$-model $\mathcal{E}$ of $(M,\Phi)$ over $K$ on $A_{K,t}[s,1)$.
\end{maintheorem}

The following structural result is useful independently of the analytic estimates.

\begin{maintheorem}[Bounded rank-one pushforward realization]\label{thm:C}
Let $M$ be a quasi-constant differential module over $\mathcal{R}_{K,t}$ of rank at most $r$ and highest break at most $b$. After a finite coefficient extension $K'/K$, there are covers $f_j$ and rank-one quasi-constant modules $\LL_j$ on their sources such that $M_{K'}$ is a differential direct summand of
\[
 \bigoplus_j f_{j*}\LL_j,
 \qquad
 \sum_j\deg f_j\le r(p^{2r}-1),
 \qquad
 b(\LL_j)\le r(p^{2r}-1)b.
\]
If $M$ is absolutely irreducible, one rank-one pushforward suffices. Each cover factors as a wild map of degree $p^\nu$, itself a tower of degree-$p$ steps, followed by two tame maps, as in \cref{prop:induced}.
\end{maintheorem}

No bound on $[K':K]$ is claimed or needed in \cref{thm:C}. The degree bounds concern the covers in the curve variable, and hence the rank of the ambient pushforward.

\subsection{Strategy of proof}\label{sec:strategy}
The proof of \cref{thm:A} constructs a bounded ambient model carrying an auxiliary Frobenius. Dwork's argument then extends the horizontal projector onto the given module, even when that module has no prescribed Frobenius. The main quantitative ingredients are the bounded inducing index in \cref{lem:tameindex}, the uniform wild collars and compatible Frobenius construction in \cref{thm:collar,prop:wildfrob}, and the coefficient-restriction estimate in \cref{lem:coefficientrestriction}, whose cost is independent of the extension degrees. These are combined with Matsuda's decomposition, Dwork's extension argument, and orthonormal bases on circles as follows.
\begin{enumerate}
\item Matsuda's decomposition reduces the general case to quasi-constant modules and unipotent factors; see \cref{subsec:matsuda}. The latter have constant nilpotent connection matrices, and direct sums and tensor products combine the models.
\item Finite inertia representations realize absolutely irreducible quasi-constant modules, after coefficient extension, as differential summands of rank-one pushforwards, with cover degrees and ramification bounded in terms of $r,b,p$; see \cref{sec:induction}.
\item Pulita's rank-one normal forms and explicit tame and wild pushforwards give a global-basis connection witness carrying an auxiliary Frobenius on a uniformly sized annulus; see \cref{sec:rankone,sec:push}. The compatible lifts and Taylor comparison in \cref{prop:wildfrob} provide the wild part of this structure.
\item Restriction of scalars to $C_k=\operatorname{Frac}W(k)$, followed by extension of scalars to $K$, gives a bounded ambient model containing the original module as a differential direct summand. The coefficient norm comparison costs at most a factor $p$, independently of the extension degrees; see \cref{subsec:descent}.
\item Dwork's argument extends the horizontal projector, and orthonormal bases on circles transfer the connection bound to its image; see \cref{sec:dwork,subsec:mainproofs}.
\end{enumerate}
For \cref{thm:B}, the bounded ambient model carries an auxiliary Frobenius for a power of the prescribed lift. The extension of horizontal maps then extends the prescribed Frobenius and its inverse to the extracted model, with the same radius and connection bound; see \cref{prop:prescribedextension,subsec:mainproofs}.

After the rank-one realization in \cref{sec:induction}, \cref{sec:frob} gives the Frobenius conventions and finite \'etale base-change statements used in \cref{sec:rankone,sec:push}. We prove the extension and circle-basis results in \cref{sec:dwork}, then carry out Matsuda reduction and the final assembly in \cref{sec:assembly}.

\begingroup
\renewcommand{\addcontentsline}[3]{}
\section*{Acknowledgements}
This work forms part of the author's doctoral thesis at UC San Diego.
The author thanks their advisor, Kiran Kedlaya, for suggesting this problem and for many insights into its solution. The author also thanks Andrea Pulita for detailed discussions of \cref{sec:rankone,sec:push}.

\section*{Declaration of AI use}
The project and its original proof strategy were developed by the author between 2023 and 2026, and the author wrote the manuscript. Large language models (Claude Opus 5.5 and GPT6-Sol) were used to revise the wording of the author's draft for clarity and style, and to review mathematical arguments. This review identified an error in \cref{lem:rankonechangelift} and suggested a correction. It also identified a gap in an earlier proof of \cref{lem:tameindex} and suggested the Frattini-quotient argument using Burnside's theorem employed there. The author has checked all statements and proofs and takes full responsibility for the content of the paper.
\endgroup

\section{Induction from rank one}\label{sec:induction}
We realize quasi-constant modules as differential summands of rank-one pushforwards with bounded cover degrees and ramification. These bounds provide the inducing data for the analytic construction.
\subsection{Monodromy representations and splitting}\label{subsec:monodromy}
Write $\mathcal{R}=\mathcal{R}_{K,t}$. Let $S/\mathcal{R}$ be a finite Galois formally \'etale cover with group $G$ and constant field $K'=H^0_\nabla(S)$. For a differential module $M$ trivialized by $S$, put
\[
 V_S(M)=H^0_\nabla(S\otimes_{\mathcal{R}}M),\qquad
 \dim_{K'}V_S(M)=\rk(M).
\]
The natural action of $G$ on $V_S(M)$ is semilinear: $g(av)=g(a)g(v)$ for $a\in K'$. This construction gives an equivalence with finite-dimensional semilinear $G$-representations over $K'$, with inverse
\[
 V\longmapsto(S\otimes_{K'}V)^G,
\]
where the connection acts on the $S$ factor; see \cite[Lemma~4.3.2]{Ked05}. Inertia fixes $K'$, so its action is linear.

Under this correspondence, rank becomes dimension and differential direct summands become representation-theoretic direct summands. Pullback along an intermediate cover corresponds to restriction, and pushforward corresponds to induction. These operations are semilinear for $G$ and become the usual linear operations on inertia.

\begin{lemma}[Absolute splitting of quasi-constant modules]\label{lem:absolutesplitting}
Let $M$ be a quasi-constant differential module over $\mathcal{R}_{K,t}$ of rank at most $r$ and highest break at most $b$. There is a finite coefficient extension $K'/K$ such that
\[
 M_{K'}\simeq\bigoplus_i \mathcal{F}_i,
\]
where the $\mathcal{F}_i$ are absolutely irreducible quasi-constant modules satisfying
\[
 b(\mathcal{F}_i)\le b,\qquad \sum_i\rk(\mathcal{F}_i)\le r.
\]
\end{lemma}
\begin{proof}
By the monodromy correspondence and the semilinear Maschke theorem \cite[Section~4.2]{Ked05}, quasi-constant modules are semisimple and remain so after finite coefficient extension. The horizontal endomorphism algebra $\mathcal A=H^0_\nabla(\End M)$ is a finite-dimensional semisimple $K$-algebra; choose a finite splitting field $K'/K$ for it.

Horizontal sections commute with finite coefficient extension, by expanding in a constant basis, so
\[
 H^0_\nabla(\End M_{K'})\cong\mathcal A\otimes_KK'.
\]
Let $(e_i)$ be a complete set of orthogonal primitive idempotents of this split algebra, and put $\mathcal{F}_i=e_iM_{K'}$. Then $M_{K'}=\bigoplus_i\mathcal{F}_i$ and
\[
 H^0_\nabla(\End \mathcal{F}_i)
 \cong e_i(\mathcal A\otimes_KK')e_i=K'e_i\cong K'.
\]
By the same commutation, $H^0_\nabla(\End \mathcal{F}_{i,K''})\cong K''$ for every finite extension $K''/K'$. Since $\mathcal{F}_{i,K''}$ is semisimple, a nontrivial decomposition would give a nontrivial horizontal idempotent; hence each $\mathcal{F}_i$ is absolutely irreducible. Ranks add under direct sums, and coefficient extension does not increase the highest breaks, so
\[
 \sum_i\rk(\mathcal{F}_i)=\rk(M)\le r,\qquad b(\mathcal{F}_i)\le b.\qedhere
\]
\end{proof}

\begin{remark}[Finite realization of local monodromy]\label{rem:finiterealization}
Let $M$ be quasi-constant. Choose a finite Galois trivializing residue extension $E/k((t))$. Its maximal unramified subextension is $\ell((t))$ for a finite extension $\ell/k$, and $E/\ell((t))$ is totally ramified; see \cite[Chapter~IV]{SerreLF}. After the corresponding unramified coefficient extension, restrict the action on horizontal sections to $I=\operatorname{Gal}(E/\ell((t)))$. This group fixes the constants, so its action is linear. Replacing $E$ by the fixed field of the kernel makes the action faithful; hence its highest group break equals the module break, $b(I)=b(M)$. A further finite coefficient extension may be used to split the representation.
\end{remark}

For a finite totally ramified Galois residue extension $E/F$ with group $I$, a subgroup $H\le I$ corresponds to the cover with residue extension $E^H/F$. We call this cover the map associated with the inclusion $H\le I$.

\subsection{Ramification under subgroups and induction}
The inducing characters and intermediate wild steps arise from subgroups of inertia. The following estimate controls their breaks by the subgroup index.
Let $E/F$ be a finite totally ramified Galois extension with group $I$, and let $H\le I$. Lower ramification groups satisfy $H_x=H\cap I_x$ \cite[Chapter~IV, Section~1]{SerreLF}. Write
\[
 \varphi_G(y)=\int_0^y[G:G_x]^{-1}\,dx
\]
for the Herbrand function. For a Galois group $G$ equipped with this ramification filtration, $b(G)$ denotes its highest upper-numbering break, with $b(G)=0$ when there is no positive break.

\begin{lemma}\label{lem:subgroupbreak}
If $[I:H]=n_H$, then $b(H)\le n_H\,b(I)$. Every representation of $H$ has highest break at most $b(H)$.
\end{lemma}
\begin{proof}
For $x\ge0$,
\[
 [I:I_x]\le[I:H_x]=n_H[H:H_x],
\]
so $\varphi_H(y)\le n_H\varphi_I(y)$. The last lower break of $H$ is no larger than that of $I$. Applying these two observations gives $b(H)\le n_H b(I)$. A representation factors through a quotient of $H$, and upper numbering is compatible with quotients \cite[Chapter~IV, Section~3]{SerreLF}.
\end{proof}

The next formula will track the highest break at each wild pushforward step in \cref{prop:wild}.
\begin{lemma}\label{lem:inductionbreak}
Let $L/F$ be a totally ramified cyclic extension of degree $p$ and break $m$, and let $V$ be a nonzero representation of $G_L$ with finite inertia image. Put $B=b(V)$ and $B'=b(\Ind_{G_L}^{G_F}V)$. Then
\[
 B' =\max\left\{m,\frac{B+(p-1)m}{p}\right\}.
\]
\end{lemma}
\begin{proof}
The Herbrand function of the degree-$p$ extension is $\varphi(u)=u$ for $u\le m$ and $\varphi(u)=m+(u-m)/p$ for $u\ge m$; see \cite[Chapter~IV, Section~3]{SerreLF}. Above the break $m$, the ramification group lies in the inducing subgroup; the restriction of the induced representation is a direct sum of conjugates, so its last break there is $\varphi(B)$ if $B>m$. Up to the break $m$, the nontrivial action on cosets forces nontrivial action on the induced representation. Consequently its last break is $\max\{m,\varphi(B)\}$, which is the stated formula.
\end{proof}

\subsection{An induced ambient representation}
The character-theoretic arguments below are carried out over an algebraically closed field of characteristic zero; the finitely many representations and intertwiners involved are then realized over a finite splitting coefficient extension. For a finite group $G$, write $\Irr(G)$ for its irreducible characters, and use $\Ind_H^G$ for induction.

We use Isaacs's relative $M$-group theorem \cite[Definition~6.21 and Theorem~6.22]{Isaacs} and monomiality of irreducible representations of finite $p$-groups \cite[Chapter~6]{Isaacs}. We also use Schur's lemma \cite{SerreRep}.

To prove \cref{thm:C}, we seek a rank-one induced representation containing the given irreducible representation as a direct summand. Bounding its inducing index will bound both the cover degree and the break of the inducing character.
Let $\theta:I\to\GL_d$ be faithful and irreducible, where $I$ is a finite inertia group and $P$ is its wild inertia subgroup. The quotient $I/P$ is cyclic of order prime to $p$. In particular, it is supersolvable, so \cite[Theorem~6.22]{Isaacs} implies that $I$ is a relative $M$-group with respect to $P$. By \cite[Definition~6.21]{Isaacs}, there are a subgroup $P\le H\le I$ and an irreducible character $\psi\in\Irr(H)$ such that
\[
 \theta=\Ind_H^I\psi,\qquad \psi|_P\in\Irr(P).
\]
By the Schur--Zassenhaus theorem \cite[Theorem~1]{ConSZ}, choose a complement $T$ with $H=P\rtimes T$, and set
\[
 T_0=\ker(T\longrightarrow\Aut P),\qquad H_0=P\times T_0.
\]
Since $P$ is monomial, there are $Q\le P$ and a linear character $\lambda$ of $Q$ such that $\psi|_P=\Ind_Q^P\lambda$. Schur's lemma shows that $T_0$ acts on $\psi$ through a scalar character $\lambda_0$. Hence
\begin{equation}\label{eq:restriction}
 \psi|_{H_0}=\psi|_P\boxtimes\lambda_0
 =\Ind_{Q\times T_0}^{H_0}(\lambda\boxtimes\lambda_0).
\end{equation}
It follows by induction in stages and Frobenius reciprocity that $\psi$ is a direct summand of $\Ind_{Q\times T_0}^H(\lambda\boxtimes\lambda_0)$, and $\theta$ is a direct summand after induction to $I$.

\begin{lemma}\label{lem:degrees}
Writing $d=d'p^{v_p(d)}$ with $(d',p)=1$, the three maps associated to
\[
 Q\times T_0\subseteq H_0\subseteq H\subseteq I
\]
have degrees $p^{v_p(d)}$, $[T:T_0]$, and $d'$, respectively.
\end{lemma}
\begin{proof}
The degree of the irreducible $P$-representation $\psi|_P$ is a power of $p$. Since $d=[I:H]\dim\psi$ and $[I:H]$ is prime to $p$, that degree is $p^{v_p(d)}$. The other two indices follow immediately.
\end{proof}

\begin{lemma}\label{lem:tameindex}
One has $[T:T_0]\le p^{d+v_p(d!)}-1\le p^{2d}-1$.
\end{lemma}
\begin{proof}
Let $\operatorname{Frat}(P)$ be the Frattini subgroup. The kernel of
$\Aut(P)\to\Aut(P/\operatorname{Frat}(P))$ is a $p$-group by Burnside's theorem \cite[Theorem~5.1.4]{Gor68}. Since $T/T_0$ has order prime to $p$, its faithful action on $P$ remains faithful on the Frattini quotient.

The faithful representation $\theta|_P$ is a direct sum of monomial irreducibles, so $P$ is a group of monomial $d\times d$ matrices. Let $D$ be its diagonal subgroup. The finite abelian $p$-group $D$ has at most $d$ generators, while $P/D$ embeds in the symmetric group $S_d$ and has order at most $p^{v_p(d!)}$. Lifting generators from $P/D$ gives
\[
 d_P:=\dim_{\mathbb{F}_p}P/\operatorname{Frat}(P)\le d+v_p(d!)\le2d.
\]
Thus $T/T_0$ embeds in $\GL_{d_P}(\mathbb{F}_p)$. If $d_P=0$, this quotient is trivial and the claimed bound holds. Otherwise, let $A$ be the image of a generator. By Cayley--Hamilton, the algebra $\mathbb{F}_p[A]$ has dimension at most $d_P$. The distinct powers of $A$ are nonzero elements of that algebra, so
\[
 [T:T_0]=\operatorname{ord}(A)\le p^{d_P}-1
       \le p^{d+v_p(d!)}-1\le p^{2d}-1.
\]
\end{proof}

\begin{remark}[Image orders for representations of wild inertia]\label{rem:earliertameindex}
The proof of \cref{lem:tameindex} controls the tame action through the Frattini quotient, without requiring a bound on the order of the wild inertia image. Such an order bound cannot in general be obtained from the dimension and highest break alone.

The order estimate in \cite[Proposition~2.3.9]{Ked05}, as stated in the cited arXiv version, does not hold for arbitrary finite-image representations of wild inertia equipped with the upper numbering inherited from the base field. A rank-one family already has bounded highest break and unbounded image order.

Let $k$ be algebraically closed of characteristic $p$, put $F=k((t))$, and, for each $a\ge1$, take the tame extension
\[
 F_a=k((u)),\qquad t=u^{e_a},\qquad e_a=p^a-1.
\]
The absolute wild inertia groups of $F$ and $F_a$ identify; denote their common group by $P_F$. The length-$a$ Artin--Schreier--Witt vector $(u^{-1},0,\ldots,0)$ defines a character $\chi_a:G_{F_a}\to\mu_{p^a}\subset\overline{\Q}_p^{\times}$ of order $p^a$. Indeed, its reduction to length one is the nonzero Artin--Schreier class of $u^{-1}$. The reduced-Witt conductor formula gives highest break $p^{a-1}$ in the $u$-normalization \cite[Theorem~4.9 and Section~4.2.3]{CP09}. Its restriction to $P_F$ still has image of order $p^a$, since the tame quotient has no nontrivial finite $p$-group quotient. Tame base change multiplies positive upper breaks by $e_a$ \cite[Chapter~IV, Section~3]{SerreLF}, so
\[
 b_F(\chi_a|_{P_F})=\frac{p^{a-1}}{p^a-1}\le1,
 \qquad |\chi_a(P_F)|=p^a.
\]
An abelian quotient of $P_F$ need not arise from a finite abelian extension of $F$, so Hasse--Arf does not give integral $F$-breaks for these arbitrary wild-inertia characters. They need not extend to rank-one representations of the full inertia group of $F$. The example therefore concerns the general wild-inertia order estimate. It does not contradict \cref{lem:tameindex}, nor the conclusion of \cite[Remark~6.1.6]{Ked05}, which invokes that estimate as one alternative but also gives a separate argument under additional hypotheses.
\end{remark}

\begin{proposition}\label{prop:induced}
Over a splitting coefficient field, an absolutely irreducible quasi-constant module $M$ of rank $d$ and highest break at most $b$ is a differential direct summand of $(h_2\circ h_1\circ g)_*\LL$, with $\LL$ rank one. The maps $h_1,h_2$ are tame, and their degrees and the degree of the wild map $g$ satisfy
\[
 \deg g=p^{v_p(d)},\qquad \deg h_1=[T:T_0]\le p^{2d}-1,
 \qquad \deg h_2=d/p^{v_p(d)}.
\]
In particular,
\[
 \deg(h_2\circ h_1\circ g)=d[T:T_0]\le d(p^{2d}-1),
 \qquad b(\LL)\le d(p^{2d}-1)b.
\]
The wild map factors into exactly $\nu=v_p(d)$ degree-$p$ steps.
\end{proposition}
\begin{proof}
Apply \eqref{eq:restriction}, \cref{lem:degrees,lem:tameindex}, and the correspondence between quasi-constant differential modules and finite local monodromy representations. The break estimate is \cref{lem:subgroupbreak} applied to $Q\times T_0\le I$, whose index is $d[T:T_0]$. Every subgroup of a finite $p$-group is subnormal, so choose a chain from $Q$ to $P$ with each inclusion normal of index $p$. Adjoining the central factor $T_0$ preserves this property. The associated wild map therefore factors into exactly $v_p([P:Q])=v_p(d)$ cyclic degree-$p$ steps.
\end{proof}

For later use, put
\[
 B_*(d,b,p)=d(p^{2d}-1)b.
\]

\begin{proof}[Proof of \cref{thm:C}]
Use \cref{lem:absolutesplitting,rem:finiterealization} to split $M$, after a finite coefficient extension, as a sum of absolutely irreducible quasi-constant modules $M_j$ of ranks $d_j$, counted with multiplicity. Then $\sum_j d_j\le r$ and $b(M_j)\le b$. Apply \cref{prop:induced} to each $M_j$ and take a common finite coefficient extension for the resulting data. The direct-summand assertions combine under direct sums, and
\[
 \sum_j\deg f_j\le\sum_j d_j(p^{2d_j}-1)
       \le r(p^{2r}-1),\qquad
 b(\LL_j)\le d_j(p^{2d_j}-1)b\le r(p^{2r}-1)b.\qedhere
\]
\end{proof}

\subsection{The tame twist and differential summands}\label{subsec:tametwist}
For the analytic construction, we now separate the scalar tame factor from the inducing module in \cref{prop:induced}, so that the wild tower starts from a purely wild rank-one module.
\begin{lemma}\label{lem:twist}
There is a tame character $\chi$ of $H$ such that
\[
 \Ind_{Q\times T_0}^H(\lambda\boxtimes\lambda_0)\otimes\chi^{-1}
 \simeq\Ind_{Q\times T_0}^H(\lambda\boxtimes1).
\]
\end{lemma}
\begin{proof}
Write $T=\langle\gamma\rangle$ and put $n_1=[T:T_0]$, so $T_0=\langle\gamma^{n_1}\rangle$. Choose $\zeta$ with $\zeta^{n_1}=\lambda_0(\gamma^{n_1})$ and $\zeta^{|T|}=1$. Sending $\gamma$ to $\zeta$ extends $\lambda_0$ to $T$. Inflate this character along $H\to T$, and use the projection formula for induction.
\end{proof}

The character $\lambda\boxtimes1$ is purely wild. Let $M_\psi$ be the differential module corresponding to $\psi$, let $\mathcal T_1$ correspond to the tame character $\chi$ of \cref{lem:twist}, and let $\widetilde\LL$ correspond to $\lambda\boxtimes1$. The inducing module $\LL$ in \cref{prop:induced} is related to it by
\[
 \LL\simeq\widetilde\LL\otimes(h_1\circ g)^*\mathcal T_1.
\]
The original constituent is $h_{2*}M_\psi$, and
\[
 h_1^*(M_\psi\otimes\mathcal T_1^\vee)\simeq g_*\widetilde\LL.
\]
Consequently the differential ambient module is
\begin{equation}\label{eq:ambient}
 V=h_{1*}h_1^*(M_\psi\otimes\mathcal T_1^\vee)
       \simeq h_{1*}g_*\widetilde\LL.
\end{equation}
The direct-summand assertion already proved by Frobenius reciprocity shows that the original constituent is a differential direct summand of $h_{2*}(V\otimes\mathcal T_1)$. The tame twist has exponent in $\mathbb{Z}_{(p)}$ and connection norm at most one.

\section{Frobenius conventions and \'etale base change}\label{sec:frob}
We fix matrix conventions and record the elementary Frobenius structures and finite \'etale base-change statements needed for the explicit pushforwards.
\subsection{Matrix conventions}
For $m\ge1$, put $q=p^{m}$, $\phi=\sigma^{m}$, and
\[
 \Phi^{(m)}=\Phi\circ\sigma^*\Phi\circ\cdots\circ(\sigma^{m-1})^*\Phi.
\]
We write a basis as a row vector $\ee$. If $\nabla_{\del_t}(\ee)=\ee A$ and $\Phi^{(m)}(\phi^*\ee)=\ee \FM $, then for a column vector $x$,
\[
 \nabla_{\del_t}(\ee x)=\ee(Ax+\del_t x),\qquad
 \Phi^{(m)}(\ee x)=\ee \FM \phi(x).
\]
For the standard lift $\phi(t)=t^q$, horizontality is
\begin{equation}\label{eq:horizontal}
 A\FM +\del_t(\FM )=q\FM \phi(A).
\end{equation}
Under $\ee'=\ee U$, the matrices become
\[
 A'=U^{-1}AU+U^{-1}\del_t(U),\qquad
 \FM '=U^{-1}\FM \phi(U).
\]
For a general lift $\tau$, replace $q$ in \eqref{eq:horizontal} by $\del_t(\tau(t))/\tau(t)$.

The connection and Frobenius live on different annuli. At a radius $\rho\ge s^{1/q}$, the source basis for Frobenius is the pullback of the basis at $\rho^q$.

\subsection{Elementary Frobenius structures}
\begin{remark}\label{rem:elementaryfrob}
Let $q=p^m$ and $\phi=\sigma^m$. A tame rank-one model $\nabla_{\del_t}e=a e$, with $a\in\mathbb{Z}_{(p)}$, admits a level-$m$ Frobenius structure exactly when $(q-1)a\in\mathbb{Z}$. In that case
\[
 \Phi(\phi^*e)=t^{(q-1)a}e
\]
defines one, with analytic inverse on every relevant annulus. Every such exponent admits some level because its denominator is prime to $p$.

A unipotent Jordan model $\nabla_{\del_t}e_j=e_{j-1}$, with $\nabla_{\del_t}e_0=0$, carries the structure
\[
 \Phi(\phi^*e_j)=q^j e_j.
\]
The displayed maps are horizontal by \eqref{eq:horizontal}. Taking a common multiple of levels supplies auxiliary structures for direct sums and tensor products of these models without changing their model annulus.
\end{remark}

\subsection{A compatible base-change statement}
\begin{lemma}\label{lem:frobpush}
Let $h:X\to Y$ be a finite \'etale map of annuli. Choose outer subannuli $Y'\subseteq Y$ and $X'=h^{-1}(Y')\subseteq X$, and Frobenius maps fitting into a cartesian square
\begin{equation}\label{eq:frobenius-square}
\begin{tikzcd}[column sep=large,row sep=large]
 X' \arrow[r,"\phi_X"] \arrow[d,"h'"'] & X \arrow[d,"h"] \\
 Y' \arrow[r,"\phi_Y"'] & Y .
\end{tikzcd}
\end{equation}
Here $h'=h|_{X'}$. If $\mathcal{E}$ has a horizontal Frobenius isomorphism $\phi_X^*\mathcal{E}\xrightarrow{\sim}\mathcal{E}|_{X'}$, then $h_*\mathcal{E}$ has the induced horizontal Frobenius isomorphism on $Y'$. It models the corresponding pushforward over the Robba ring.
\end{lemma}
\begin{proof}
Finite \'etale base change gives
\[
 \phi_Y^*h_*\mathcal{E}\simeq h'_*\phi_X^*\mathcal{E}
 \xrightarrow{\ h'_*\Phi_{\mathcal{E}}\ }h'_*(\mathcal{E}|_{X'})=(h_*\mathcal{E})|_{Y'}.
\]
The connection on $h_*\mathcal{E}$ uses the unique extension of the base derivation to the finite \'etale algebra; it makes this isomorphism horizontal. Restriction of scalars commutes with restriction to smaller annuli and with passage to the Robba ring.
\end{proof}

\begin{lemma}[Finite \'etale comparison]\label{lem:etalecomparison}
Let $A\hookrightarrow R$ be an injective homomorphism of domains, and let $B,C$ be finite free \'etale $A$-algebras of the same rank. If an $A$-algebra homomorphism $\theta:B\to C$ becomes an isomorphism after extension to $R$, then $\theta$ is an isomorphism.
\end{lemma}
\begin{proof}
Choose bases of $B,C$, and let $T$ be the matrix of $\theta$. The trace pairings agree under $\theta$ after extension to $R$, hence over $A$ by injectivity. Their matrices therefore satisfy $G_B=T^{\mathsf t}G_CT$, so
\[
 \operatorname{disc}(B)=(\det T)^2\operatorname{disc}(C).
\]
Both discriminants are units because the algebras are \'etale. Thus $\det T$ is a unit, and $\theta$ is an isomorphism.
\end{proof}

\begin{corollary}[Cartesianness on a collar]\label{cor:cartesiancollar}
Consider a commuting square as in \eqref{eq:frobenius-square}, with finite \'etale vertical maps of the same degree. Suppose the section algebra of $X'$ and the pullback algebra are finite free over $\Gamma(Y',\OO)$. If the comparison map is analytic on the whole collar and the square becomes cartesian over the Robba ring, then it is cartesian on the whole collar.
\end{corollary}
\begin{proof}
Apply \cref{lem:etalecomparison} to the comparison between the pullback algebra and the algebra of $X'$, using the injection $\Gamma(Y',\OO)\hookrightarrow\mathcal{R}$. This gives the analytic inverse on the stated collar.
\end{proof}

\section{Rank-one models}\label{sec:rankone}
We construct rank-one normal-form models with uniform connection bounds and an auxiliary Frobenius collar. These models supply the initial modules for the pushforwards in \cref{sec:push}.
\subsection{Coefficient fields and Lubin--Tate parameters}\label{subsec:coefficients}
The rank-one models and the lifted covering equations use a common Witt coefficient field with finitely many Lubin--Tate parameters.
Fix once and for all a Lubin--Tate tower for a series over $\mathbb{Z}_p$ with uniformizer $p$, with first parameter $\pi_0^{p-1}=-p$. For $j\ge0$, put $F_j=\Q_p(\pi_j)$; this field contains $\pi_0,\ldots,\pi_j$ and is totally ramified over $\Q_p$ \cite[Section~1.5]{Pul07}. For later use, write
\[
 j_*(B)=\left\lfloor\log_p\max\{1,B\}\right\rfloor
\]
for the Lubin--Tate cutoff associated with a break bound $B$.

\begin{remark}[Witt coefficient fields and Frobenius]\label{rem:coefficientlift}
For a finite residue extension $\ell/k$, put $C_\ell=\operatorname{Frac}W(\ell)$. The totally ramified field $F_j$ is linearly disjoint from $C_\ell$ over $\Q_p$. Witt Frobenius therefore extends to an isometric automorphism of $C_\ell F_j$ fixing $F_j$ pointwise. This is the coefficient lift used for the explicit models.

Any isometric absolute Frobenius lift on a complete coefficient field with perfect residue field $\ell$ preserves $C_\ell$ and restricts to Witt Frobenius there. Indeed, it takes each Teichm\"uller representative $[x]$ to $[x^p]$, fixes $p$, and is continuous; apply this to the expansion $\sum_{j\ge0}p^j[x_j]$. In particular $\sigma_K|_{C_k}=\sigma_{C_k}$ for a prescribed lift on $K$. The final descent restricts scalars to $C_k$, so no extension of $\sigma_K$ to a splitting field is needed.
\end{remark}

\subsection{Normal form and the connection estimate}\label{subsec:rankonenormal}
We use the rank-one classification of Pulita \cite[Theorem~3.1]{Pul07}. For earlier work on the differential equations associated with Artin--Schreier--Witt coverings and their local indices, see Matsuda \cite{Mat95}. For the compatible Lubin--Tate parameters fixed in \cref{subsec:coefficients}, one has
\[
 |\pi_j|=p^{-1/(p^j(p-1))},
\]
and we work over a coefficient field containing the finitely many parameters required. We write $w$ for the coordinate; in \cref{sec:push} it is $w_0$. After a finite residue extension removing the constant Witt part, a rank-one quasi-constant module is described by a tame exponent $a_0\in\mathbb{Z}_{(p)}$ and a reduced residue Witt vector of length $n+1$
\[
 \bar{\ff}^-=(\bar f_0,\ldots,\bar f_n),\qquad
 \bar f_i\in w^{-1}k[w^{-1}].
\]
These polynomials describe the wild part. Put $\bar d_i=\deg_{w^{-1}}\bar f_i$, omitting zero components. The break of the reduced character is
\begin{equation}\label{eq:rankonebreak}
 B=\max_i p^{n-i}\bar d_i.
\end{equation}
Here $B$ is the Swan conductor, given by the reduced-Witt formula of \cite[Theorem~4.9 and Section~4.2.3, equation~(4.2.22)]{CP09}. Over a perfect residue field this agrees with the classical ramification-group Swan conductor \cite[Proposition~7.1 and its proof]{Bor02}, and with the differential Swan conductor \cite[Theorem~7.1]{CP09}, which in rank one is the single differential break.\footnote{For the earlier result over a finite residue field, see \cite[Section~1, pp.~26--27]{Bry83}. The conductor there is the Artin conductor, equal to $B+1$ for a nontrivial wild rank-one character.} For a nontrivial wild character, choose the length so that $n+1\le1+\lfloor\log_p B\rfloor$.

\paragraph{The Witt-monomial lift.}
Give a monomial $w^{-k}$ in component $i$ weighted pole order $kp^{n-i}$. Express $\bar{\ff}^-$ as a finite Witt sum of terms $V^i([\bar a w^{-k}])$, successively matching its components from component zero onward. Witt addition and subtraction preserve the weighted pole bound, since their component polynomials are homogeneous for these weights. At component $i$ there are at most $\lfloor B/p^{n-i}\rfloor$ monomials, so at most
\[
 J(B,p)=\sum_{j=0}^{\lfloor\log_p B\rfloor}\lfloor B/p^j\rfloor
\]
terms are needed, all with $kp^{n-i}\le B$. Lift each $\bar a$ to its Teichm\"uller representative $a$ and take the same Witt sum in characteristic zero. This gives
\[
 \ff^-=(f_0,\ldots,f_n),\qquad f_i\in w^{-1}\OO_K[w^{-1}],
\]
with the same weighted pole bound; its components need not be the componentwise Teichm\"uller lifts. If $d_i=\deg_{w^{-1}}f_i$, then $p^{n-i}d_i\le B$, with equality for at least one component when the wild part is nontrivial. For $B<1$, that part is trivial and the construction is empty.

Pulita's notation $\mathrm e_{p^n}(\ff^-,1)$ denotes the Witt exponential formed from the ghost components and the chosen Lubin--Tate parameters. Set $Y=w^{a_0}\mathrm e_{p^n}(\ff^-,1)$ and use the convention $\omega=\del_w\log Y$. Logarithmic differentiation gives
\begin{equation}\label{eq:rankonecoefficient}
 \omega=a_0+\sum_{j=0}^n\pi_{n-j}
          \sum_{i=0}^{j}f_i^{\,p^{j-i}-1}\del_w(f_i).
\end{equation}
\begin{lemma}\label{lem:rankone}
The reduced normal form of a rank-one quasi-constant module of break $B$ defines a model on every annulus. For every $0<c<1$ it has a global-basis witness for $(c^{1/(B+1)},c^{-1})$-boundedness.
A tame rank-one module is $(s,1)$-bounded with a global-basis witness for every $s\in(0,1)$.
For a purely wild normal form with exponent $a_0=0$ and break at most $B$, the stronger estimate is
\begin{equation}\label{eq:smallrankone}
 \begin{gathered}
 \omega=\sum_{1\le k\le B}c_kw^{-k},\qquad |c_k|\le\kappa,\\
 |\omega|_\rho\le\kappa\rho^{-B},\qquad
 \kappa=|\pi_0|^{1/p^{j_*(B)}}<1.
 \end{gathered}
\end{equation}
\end{lemma}
\begin{proof}
Each summand in \eqref{eq:rankonecoefficient} is a Laurent polynomial with integral coefficients and pole order at most $p^{j-i}d_i\le B$. Since $|a_0|\le1$ and $|\pi_j|\le1$, the connection coefficient $\omega$ satisfies
\[
 |\omega|_\rho\le\max\{1,\rho^{-B}\}.
\]
For $\rho\ge c^{1/(B+1)}$, this is at most $c^{-B/(B+1)}\le c^{-1}$. A tame module has constant coefficient $a_0\in\mathbb{Z}_{(p)}$, of norm at most one.

If $a_0=0$, every term has strictly negative exponent. A nonzero term indexed by $i\le j$ in \eqref{eq:rankonecoefficient} has $p^{n-i}d_i\le B$, hence $n-j\le j_*(B)$ and $|\pi_{n-j}|\le\kappa$. Its remaining coefficients are integral, so \eqref{eq:smallrankone} follows. This also covers a presentation with initial zero Witt components; when $B<1$, the purely wild part is trivial and $\omega=0$.
\end{proof}

The dual module has coefficient $-\omega$ and satisfies the same estimates.

\Needspace{5\baselineskip}
\subsection{Frobenius for the purely wild factor}
\begin{lemma}[Uniform rank-one Frobenius collar]\label{lem:rankonefrob}
Use the isometric absolute coefficient lift of \cref{rem:coefficientlift}, which acts by Witt Frobenius and fixes the selected Lubin--Tate parameters, and the variable lift $w\mapsto w^p$. Purely wild rank-one modules of break at most $B$ admit normal-form models carrying a level-one Frobenius isomorphism on a collar whose inner radius depends only on $B,p$. The inverse is analytic on the same Frobenius subannulus.
\end{lemma}
\begin{proof}
For $B<1$, use the constant model. Otherwise remove any integral exponent and use the Witt-monomial lift constructed in \cref{subsec:rankonenormal}.

The additivity of the Witt exponential gives an exact product of elementary factors
\[
 Y(w)=\prod E_j(a w^{-k}),\qquad
 E_j(X)=\exp\left(\sum_{v=0}^j\frac{\pi_{j-v}}{p^v}X^{p^v}\right),
 \qquad kp^j\le B;
\]
see \cite[Sections~2.2 and~3.2]{Pul07}. Since $\sigma_K(a)=a^p$ and the parameters are fixed, the Frobenius quotient is a product of
\[
 \Theta_j(a w^{-k}),\qquad \Theta_j(X)=E_j(X^p)/E_j(X).
\]
By \cite[Theorem~2.5]{Pul07}, these functions and their inverses are overconvergent. Choose $R_\Theta(B,p)>1$ strictly inside their convergence discs for all $j\le\lfloor\log_p B\rfloor$. Choose $s_*<1$ with $s_*^{-k}\le R_\Theta(B,p)$ for every possible $k$. All the substitutions, their finite product, and its inverse are analytic on $A[s_*,1)$. The finite choices of $j,k$ show that $s_*$ depends only on $B,p$.

Logarithmic differentiation proves horizontality of the quotient. Taking the differential-model radius at least $s_*^p$ places its Frobenius subannulus inside $A[s_*,1)$.
\end{proof}

\section{Effective pushforwards of modules}\label{sec:push}
We estimate pushforwards through a tower of wild degree-$p$ steps followed by two tame power maps, as in \cref{fig:cover-tower}. In the final assembly, \cref{prop:geometric} realizes the inducing cover of \cref{prop:induced} in this form after a finite coefficient extension.

\begin{figure}[!ht]
\centering
\begin{tikzcd}[column sep=large]
 A_{w_0} \arrow[r,"g_0","p"'] & \cdots \arrow[r,"g_{\nu-1}","p"']
 & A_{w_\nu} \arrow[r,"h_1","n_1"'] & A_u \arrow[r,"h_2","d'"'] & A_t .
\end{tikzcd}
\caption{Factorization of the inducing cover.}\label{fig:cover-tower}
\end{figure}

Here $A_x$ denotes a suitable outer annulus in coordinate $x$, and each arrow is labeled by its name and degree. For an absolutely irreducible constituent of rank $d$, the inducing degrees are
\[
 \nu=v_p(d),\qquad n_1=[T:T_0],\qquad d'=d/p^\nu.
\]
The wild map is $g=g_{\nu-1}\circ\cdots\circ g_0$, interpreted as the identity when $\nu=0$. The rank-one module of \cref{sec:rankone} lives on $A_{w_0}$. We use $u=w_\nu^{\,n_1}$ and $t=u^{d'}$ for the tame maps; these coordinate choices are justified in \cref{prop:geometric}. We treat the tame maps first and then the wild steps.

\subsection{Effective tame pushforward}
For a tame power map, the power basis gives both the radius change and the connection bound under pushforward.
Let $(n,p)=1$ and
\[
 f:A_{K,u}[s,1)\longrightarrow A_{K,t}[s^n,1),\qquad t=u^n.
\]
Then $\del_u=n\del_t$ on the base ring, so the pushforward connection is $\nabla_{\del_t}=n^{-1}\nabla_{\del_u}$.

\begin{lemma}\label{lem:tamebasis}
At radii $|u|=\rho$, $|t|=\rho_t=\rho^n$, the elements $1,u,\ldots,u^{n-1}$ form an orthogonal basis of the upper ring over the lower ring. In particular,
\[
 \left|\sum_{j=0}^{n-1}a_j(t)u^j\right|_\rho
 =\max_j |a_j|_{\rho_t}\rho^j.
\]
\end{lemma}
\begin{proof}
Group the Laurent series in $u$ according to its exponent modulo $n$. The summands occupy disjoint sets of Laurent exponents, and the equality follows from the Gauss norm.
\end{proof}

\begin{lemma}[Tame connection bound]\label{lem:tame}
If $M$ is $(s,N)$-bounded, then $f_*M$ is $(s^n,N')$-bounded, where
\[
 N'=\max\{1,s^{1-n}N\}.
\]
Global-basis witnesses give global-basis witnesses for this bound.
\end{lemma}
\begin{proof}
Use the basis $u^j e_l$. If $A(u)=\sum_{k=0}^{n-1}A_k(t)u^k$ is the upper connection matrix, \cref{lem:tamebasis} gives
\[
 |A_k|_{\rho_t}\le\rho^{-k}N.
\]
Multiplication by $A(u)$ on the power basis has entries of the form $A_k(t)$ or $tA_k(t)$, so its norm is at most $\rho^{1-n}N$. Differentiating $u^j$ contributes the diagonal entries $j/n$, of norm at most one. Since $|n|=1$ and $\rho\ge s$, the asserted bound follows.
\end{proof}

\begin{remark}[Frobenius under tame pushforward]\label{rem:tamefrob}
Use the compatible standard lifts $u\mapsto u^q$ and $t\mapsto t^q$, with $q=p^{m}$. In the power basis, relative Frobenius has columns $u^{qj}$, $0\le j<n$. Because $(q,n)=1$, its matrix is monomial with nonzero entries powers of $t$. The matrix and its inverse are analytic on every relevant annulus. Thus \cref{lem:frobpush} carries any auxiliary Frobenius through a tame power map without shrinking the model beyond the radius change $s\mapsto s^n$.
\end{remark}

\subsection{Effective wild pushforward}\label{subsec:wildsetup}
\subsubsection{Notation and setup}
We realize each cyclic degree-$p$ step by an explicit annular cover and estimate its pushforward connection in a power basis. The lifting construction is the degree-$p$ Kummer--Artin--Schreier deformation of Sekiguchi--Oort--Suwa \cite{SOS89}.
Set $\pi=\pi_0$, the first Lubin--Tate parameter fixed in \cref{subsec:coefficients}; thus $\pi^{p-1}=-p$. A totally ramified cyclic degree-$p$ extension of $k((x))$ has a reduced Artin--Schreier equation
\[
 y^p-y+\bar P(x^{-1})+c=0,
\]
where $\bar P\in T k[T]$ has degree $m$ prime to $p$ and $c\in k$. The positive-power part is removed by convergence in $k[[x]]$, and perfection removes negative exponents divisible by $p$. A finite residue extension removes the constant part by adjoining a root of its Artin--Schreier equation. With $c=0$ and $m\alpha-p\beta=1$, the element $y^{-\alpha}x^{-\beta}$ is a uniformizer: in the upper normalized valuation, $v(x)=p$ and $v(y)=-m$.

Let $\bar w_i$ denote residue coordinates and consider a tower with $\nu\ge0$ totally ramified degree-$p$ steps
\[
 k((\bar w_\nu))\subset k((\bar w_{\nu-1}))\subset\cdots\subset k((\bar w_0)).
\]
For $0\le i<\nu$, write the $i$th step as
\[
 y_i^p-y_i+\bar P_{i+1}(\bar w_{i+1}^{-1})=0,
\]
where $\bar P_{i+1}\in T k[T]$ has degree $m_{i+1}$ prime to $p$. Choose nonnegative integers $\alpha_{i+1},\beta_{i+1}$ with
\[
 m_{i+1}\alpha_{i+1}-p\beta_{i+1}=1.
\]
Choose coefficientwise Teichm\"uller lifts $P_{i+1}\in T\OO_K[T]$, with unit leading coefficient, and lift the equations by
\begin{equation}\label{eq:wildequation}
 (1+\pi y_i)^p=1+p\pi P_{i+1}(w_{i+1}^{-1}),\qquad
 w_i=y_i^{-\alpha_{i+1}}w_{i+1}^{-\beta_{i+1}}.
\end{equation}
After division by $p\pi$, the reduction of the first equation is the stated Artin--Schreier equation.

\subsubsection{The annular collar}
\begin{theorem}[Uniform annular collar]\label{thm:collar}
Let $s\in(0,1)$ and $s_i=s^{1/p^{\nu-i}}$. Assume
\[
 s_{i+1}^{m_{i+1}}>|\pi|^p\qquad(0\le i<\nu).
\]
Then \eqref{eq:wildequation} defines a finite \'etale degree-$p$ map
\[
 g_i:A_{K,w_i}[s_i,1)\longrightarrow A_{K,w_{i+1}}[s_{i+1},1),
\]
with $w_i$ an analytic coordinate and
\[
 |w_{i+1}|=|w_i|^p,\qquad |y_i|=|w_i|^{-m_{i+1}}.
\]
Taking the direct limit over outer collars recovers the lifted Robba-ring tower.
\end{theorem}
\begin{proof}
Fix a step and write $x=w_{i+1}$, $y=y_i$, $m=m_{i+1}$, $P=P_{i+1}$, $\alpha=\alpha_{i+1}$, $\beta=\beta_{i+1}$. Put
\[
 \mathcal Q(Y)=\frac{(1+\pi Y)^p-1}{p\pi}
      =-Y^p+\sum_{j=1}^{p-1}q_jY^j,
\]
where $\pi^{p-1}/p=-1$ and $q_j\in\OO_K$. Write $P(T)=a_mT^m+\cdots$, with $a_m\in\OO_K^\times$. On $|x|<1$, $|P(x^{-1})|=|x|^{-m}>1$. Hence a solution of $\mathcal Q(y)=P(x^{-1})$ has $|y|>1$, and comparison of the leading terms gives $|y|^p=|x|^{-m}$.

Set $z=y^{-\alpha}x^{-\beta}$ and $\upsilon=x^my^p$. Since $m\alpha-p\beta=1$,
\[
 x=z^p\upsilon^\alpha,\qquad y=z^{-m}\upsilon^{-\beta},\qquad |\upsilon|=1.
\]
Substituting in the equation and multiplying by $x^m$ gives
\[
 \mathcal P(\upsilon)=0,\qquad \mathcal P(U)=-U-a_m+z\mathcal H(z,U)\in\OO_K[z][U].
\]
Indeed, since $\alpha,\beta\ge0$ and $m\alpha-p\beta=1$, the nonleading terms become
\[
 x^my^j=z^{m(p-j)}\upsilon^{1+\beta(p-j)}\quad(j<p),\qquad
 x^{m-l}=z^{p(m-l)}\upsilon^{\alpha(m-l)}\quad(l<m).
\]
Their $\upsilon$-exponents are nonnegative and their $z$-exponents are positive, proving the polynomial assertion.

We use Hensel's lemma in the following form \cite[Theorem~7.3]{Eisenbud}: if a ring $R$ is complete with respect to an ideal $J$, and $F\in R[U]$, $a\in R$ satisfy $F(a)\in J$ and $F'(a)\in R^\times$, then $F$ has a unique root in $a+J$. Apply it with $R=\OO_K[[z]]$, $J=(z)$, and $a=-a_m$. Here $\mathcal P(-a_m)=z\mathcal H(z,-a_m)\in J$, and
\[
 \mathcal P'(-a_m)=-1+z\frac{\partial\mathcal H}{\partial U}(z,-a_m)
\]
is a unit. We obtain a unique $\upsilon(z)\in-a_m+z\OO_K[[z]]$ with $\mathcal P(\upsilon(z))=0$. Its integral coefficients give convergence on $|z|<1$, and
\[
 |\upsilon(z)+a_m|_\rho\le\rho<1,\qquad |\upsilon(z)|_\rho=1
 \quad(0<\rho<1).
\]
In particular $\upsilon(z)$ and its inverse are analytic there.

Conversely, any point $(x,y)$ of the cover has $0<|z|<1$ and $|\upsilon|=1$. Its defining equation gives $\upsilon+a_m=z\mathcal H(z,\upsilon)$, so $\upsilon\in-a_m+z\OO_{K_1}$ over a complete valued extension $K_1/K$. Hensel uniqueness over $\OO_{K_1}$ therefore identifies it with $\upsilon(z)$. Hence the analytic maps
\[
 z\longmapsto\bigl(z^p\upsilon(z)^\alpha,\ z^{-m}\upsilon(z)^{-\beta}\bigr),
 \qquad (x,y)\longmapsto y^{-\alpha}x^{-\beta}
\]
are inverse: the first satisfies the covering equation, and $m\alpha-p\beta=1$ gives the other composition. Since $|\upsilon(z)|=1$,
\[
 |x|=|z|^p,\qquad |y|=|z|^{-m}.
\]
Thus the part of the cover over $A_{K,w_{i+1}}[s_{i+1},1)$ is the annulus $A_{K,w_i}[s_i,1)$ with coordinate $w_i=z$.

The equation is monic of degree $p$ after multiplication by $-1$, so the projection is finite free of degree $p$. The collar hypothesis gives $|\pi y|<1$, and
\[
 \mathcal Q'(y)=(1+\pi y)^{p-1}
\]
is a unit. Thus the projection is \'etale. All constructions commute with restriction to smaller outer collars.
\end{proof}

\subsubsection{Orthogonality and multiplication matrices}
Fix a step, and use the lower radius $|w_{i+1}|=\rho$. Put
\[
 m=m_{i+1},\quad y=y_i,\quad v=1+\pi y,
 \quad \eta_v=|\pi|\rho^{-m/p}.
\]
On the collar $\rho^m>|\pi|^p$, one has $|y|=\rho^{-m/p}$ and $|\pi|<\eta_v<1$.

\begin{lemma}\label{lem:wildorthogonal}
The basis $1,y,\ldots,y^{p-1}$ is orthogonal over the lower circle ring:
\[
 \left|\sum_{j=0}^{p-1}a_jy^j\right|_{\rho^{1/p}}
 =\max_j |a_j|_\rho\rho^{-jm/p}.
\]
The same holds after rescaling the basis vectors by nonzero constants.
\end{lemma}
\begin{proof}
First suppose $\rho\notin|K^\times|^{\Q}$. Each nonzero lower-ring Laurent series has a unique dominant monomial. The norms of distinct nonzero summands $a_jy^j$ cannot agree: such an equality would imply that a nonzero integral power of $\rho^{1/p}$ belongs to $|K^\times|$, unless the Laurent exponents differ by $(j-j')m/p$, which is not an integer for $0<|j-j'|<p$. The ultrametric inequality is therefore an equality.

For Laurent-polynomial coefficients, continuity of Gauss norms and density of radii outside $|K^\times|^{\Q}$ give the identity at every radius in the collar. At a fixed radius, Laurent polynomials are dense in the circle ring, so the general statement follows by approximation. Rescaling does not affect orthogonality.
\end{proof}

\begin{lemma}\label{lem:mult}
If $F_v(x)$ is the matrix of multiplication by $x$ on $1,v,\ldots,v^{p-1}$, then on the collar
\[
 |F_v(x)|_\rho\le p\,|x|_{\rho^{1/p}}.
\]
The same estimate holds for an upper-ring matrix $x$, interpreted blockwise.
\end{lemma}
\begin{proof}
In the orthogonal basis $1,\pi y,\ldots,(\pi y)^{p-1}$, the weights are $1,\eta_v,\ldots,\eta_v^{p-1}$. They lie between $|\pi|^{p-1}=p^{-1}$ and $1$. Thus the analytic norm and the unweighted coordinate norm differ by a factor at most $p$. Multiplication by $x$ has analytic operator norm $|x|$, giving the displayed coordinate bound.

The change from $(\pi y)^j$ to $v^j=(1+\pi y)^j$ is the binomial matrix $(\binom{k}{j})_{j,k}$. Both it and its inverse have integral entries, so this change is an isometry for the coordinate sup norm. The same bound therefore holds in the $v$-basis. Taking the maximum over blocks proves the matrix version.
\end{proof}

\subsubsection{Comparison of logarithmic derivations}
Write $x=w_{i+1}$ and $w=w_i$. The unique extension of $\del_x$ to the upper algebra has the form
\[
 \del_x=\gamma_i\del_w,
\]
where differentiating \eqref{eq:wildequation} gives
\begin{align}
 \frac{\del_x y}{y}
  &=-\frac{x^{-1}P'_{i+1}(x^{-1})}
              {y(1+\pi y)^{p-1}},\notag\\
 \gamma_i
  &=\frac{\alpha_{i+1}x^{-1}P'_{i+1}(x^{-1})}
              {y(1+\pi y)^{p-1}}-\beta_{i+1}.\label{eq:derivationfactor}
\end{align}
Since $|1+\pi y|=1$ and $|\del_x P_{i+1}|\le|P_{i+1}|=\rho^{-m}$,
\begin{equation}\label{eq:gamma}
 |\gamma_i|_{\rho^{1/p}}\le\rho^{-m(p-1)/p}.
\end{equation}
The pushforward connection is therefore $\nabla_{i+1}=\gamma_i\nabla_i$.

\subsubsection{The induction step}
Let $M_i$ have a global basis $e_1,\ldots,e_{r_i}$, connection matrix $A_i$, and highest break $B_i$. Use $v_i^k e_l$, $0\le k<p$, as the pushforward basis. The connection matrix is the sum of the multiplication blocks $F_{v_i}(\gamma_i(A_i)_{jl})$ and the diagonal term from differentiating $v_i^k$. Differentiation gives
\[
 \frac{\del_{w_{i+1}}v_i}{v_i}
 =-\frac{\pi w_{i+1}^{-1}P'_{i+1}(w_{i+1}^{-1})}
          {1+p\pi P_{i+1}(w_{i+1}^{-1})},
\]
which belongs to the lower ring. Its action on each $v_i^k e_l$ is scalar, so the contribution is diagonal. Its norm, together with \cref{lem:mult} and \eqref{eq:gamma}, gives
\begin{equation}\label{eq:wildrecurrence}
 |A_{i+1}|_\rho\le
 \max\left\{|\pi|\rho^{-m_{i+1}},\
 p\rho^{-m_{i+1}(p-1)/p}|A_i|_{\rho^{1/p}}\right\}.
\end{equation}

\begin{proposition}[Effective wild pushforward]\label{prop:wild}
Fix $c\in[|\pi|^p,1)$, where $|\pi|^p=p^{-p/(p-1)}$. Let $M_0$ be rank one, with highest break $B_0$, and suppose it has a global-basis witness for $(c^{1/(B_0+1)},N_0)$-boundedness. Define
\[
 M_{i+1}=g_{i*}M_i,\qquad
 B_{i+1}=\max\left\{m_{i+1},\frac{B_i+(p-1)m_{i+1}}p\right\}.
\]
Then $M_\nu$ has rank $p^\nu$, highest break $B_\nu$, and a global-basis witness on $A_{K,w_\nu}[c^{1/(B_\nu+1)},1)$ with connection bound
\[
 N_\nu\le(p/c)^\nu\max\{1,N_0\}.
\]
Its basis is
\[
 \{v_{\nu-1}^{k_{\nu-1}}\cdots v_0^{k_0}e:
               0\le k_i<p,\ 0\le i<\nu\},
\]
restricted at each stage to the indicated annulus, with basis $\{e\}$ when $\nu=0$.
\end{proposition}
\begin{proof}
For $\nu=0$ the assertion is the hypothesis. Otherwise induct on $i=0,\ldots,\nu-1$. The ranks and bases follow from restriction of scalars; the break assertion is \cref{lem:inductionbreak}. If
\[
 \rho\ge c^{1/(B_{i+1}+1)},
\]
then $p(B_{i+1}+1)>B_i+1$, and hence
\[
 \rho^{1/p}\ge c^{1/(p(B_{i+1}+1))}
          \ge c^{1/(B_i+1)}.
\]
Thus the upper model is defined at the required radius. Also $m_{i+1}\le B_{i+1}$ gives
\[
 \rho^{m_{i+1}}\ge c^{m_{i+1}/(B_{i+1}+1)}>c\ge|\pi|^p,
\]
so \cref{thm:collar} applies at this step, including when $c=|\pi|^p$.

In \eqref{eq:wildrecurrence}, both radius factors are at most $c^{-1}$. Therefore
\[
 N_{i+1}\le\max\{|\pi|c^{-1},pc^{-1}N_i\}
          \le(p/c)\max\{1,N_i\}.
\]
Iteration gives the claim. All bases are global on the restricted annuli.
\end{proof}

\subsection{Frobenius through the wild tower}\label{subsec:wildfrob}
Starting from the standard Frobenius at the bottom of the wild tower, we lift it compatibly to every stage. The resulting top lift need not be the standard one used for the rank-one model. We first construct the lifts, then compare the two rank-one pullbacks, and finally apply finite \'etale base change.

Throughout this subsection the coefficient lift is isometric, fixes $\pi$ and the selected Lubin--Tate parameters, and acts by $p$th powers on Teichm\"uller coefficients.

\begin{remark}\label{rem:compatibility}
At the Robba-ring level, a Frobenius lift on the base extends uniquely to the branch inducing absolute Frobenius on a finite \'etale residue algebra; see \cite[Definition~6.1.1 and Proposition~3.5.5]{Ked05}. The resulting relative Frobenius is an isomorphism, since relative Frobenius for an \'etale morphism in characteristic $p$ is an isomorphism \cite[Tag~0EBS]{Stacks}. Once the comparison map is analytic on a specified collar, \cref{cor:cartesiancollar} supplies its analytic inverse there. In the wild coordinates below, the extended lift is generally not $w_i\mapsto w_i^q$.
\end{remark}

\begin{lemma}[One-step Frobenius lift]\label{lem:onestepfrob}
Consider a degree-$p$ cover from \cref{thm:collar}, with lower coordinate $x$, upper coordinate $w$, and equations
\[
 \mathcal Q(y)=P(x^{-1}),\qquad w=y^{-\alpha}x^{-\beta},\qquad
 \mathcal Q(Y)=\frac{(1+\pi Y)^p-1}{p\pi},\qquad
 \mu\alpha-p\beta=1,
\]
where $\mu=\deg P$. Let its lower model be $A_x[s,1)$, so its upper model is $A_w[s^{1/p},1)$. Suppose an analytic lower lift $\sigma_x$ on $A_x[s^{1/p},1)$ satisfies, for some $D\ge0$,
\[
 |\sigma_x(f)|_\rho=|f|_{\rho^p},\qquad
 \left|\frac{\sigma_x(x)}{x^p}-1\right|_\rho
       \le |\pi|\rho^{-D},\qquad
 |\pi|s^{-(p\mu+D)/p}<1.
\]
Then the Hensel branch inducing absolute Frobenius on the residue cover defines a compatible analytic lift $\sigma_w$ on $A_w[s^{1/p^2},1)$. For $\varrho\ge s^{1/p^2}$ it satisfies
\[
 \left|\frac{\sigma_w(w)}{w^p}-1\right|_\varrho
       \le |\pi|\varrho^{-pD-p(p-1)\mu},
 \qquad |\sigma_w(f)|_\varrho=|f|_{\varrho^p}.
\]
\end{lemma}
\begin{proof}
Write $\rho=\varrho^p$. Since $\pi^{p-1}=-p$,
\[
 \mathcal Q(Y)=Y-Y^p+\sum_{j=2}^{p-1}q_jY^j,
 \qquad q_j\in\pi^{j-1}\OO_K.
\]
We solve $\mathcal Q(z)=P^{\sigma_K}(\sigma_x(x)^{-1})$ near $z=y^p$. The residual satisfies
\begin{equation}\label{eq:henselresidual}
 \left|\mathcal Q(y^p)-P^{\sigma_K}(\sigma_x(x)^{-1})\right|_\rho
       \le |\pi|\rho^{-(p\mu+D)}.
\end{equation}
Indeed, $\mathcal Q(Y^p)-\mathcal Q(Y)^p$ belongs to $\pi\OO_K[Y]$ and has degree at most $p^2$, while $|y|_\rho=\rho^{-\mu/p}$. Also $P(x^{-1})^p-P^{\sigma_K}(x^{-p})$ belongs to $p\OO_K[x^{-1}]$ and has degree at most $p\mu$. Finally, substitution of $\sigma_x(x)=x^p(1+E)$, with $|E|_\rho\le|\pi|\rho^{-D}<1$, contributes at most $\rho^{-p\mu}|E|_\rho$.

The right side of \eqref{eq:henselresidual} is strictly less than one on the stated collar. There $|\pi y^p|_\rho<1$. To check the Taylor coefficients at this point, put $v_*=1+\pi y^p$, which has norm one, and write
\[
 \mathcal Q(y^p+\Delta)
       =\frac{(v_*+\pi\Delta)^p-1}{p\pi}.
\]
For $1\le j\le p$, the coefficient of $\Delta^j$ is
$p^{-1}\binom pj\pi^{j-1}v_*^{p-j}$ and has norm at most one; the linear coefficient has norm exactly one. Hensel iteration on each closed subannulus therefore gives a unique solution $z=y^p+\Delta$ with
\[
 |\Delta|_\rho\le|\pi|\rho^{-(p\mu+D)}<1.
\]
Uniqueness makes these analytic solutions agree on overlaps. Set
\[
 \sigma_w(y)=z,\qquad
 \sigma_w(w)=z^{-\alpha}\sigma_x(x)^{-\beta}.
\]
This extends the lower lift and preserves the defining equations. Compatibility with the coordinate maps in \cref{thm:collar} follows from their uniqueness: the two analytic expressions agree at every point over every complete valued extension, and hence agree as analytic functions on each reduced closed subannulus. Moreover,
\[
 \left|\frac{\Delta}{y^p}\right|_\rho
       \le|\pi|\rho^{-D-(p-1)\mu}.
\]
Integer powers of $1+\epsilon$, with $|\epsilon|<1$, do not increase the perturbation norm. Applying this to the formula for $\sigma_w(w)$, and using $\rho=\varrho^p$, gives the asserted bound.

For completeness, put $E_w=\sigma_w(w)/w^p-1$. For every Laurent polynomial $f$,
\[
 |\sigma_w(f)-f^{\sigma_K}(w^p)|_\varrho
       \le |E_w|_\varrho\,|f|_{\varrho^p}.
\]
The standard substitution on the right is isometric, and $|E_w|_\varrho<1$, so
\begin{equation}\label{eq:liftisometry}
 |\sigma_w(f)|_\varrho=|f|_{\varrho^p}.
\end{equation}
Laurent truncations converge uniformly on each closed subannulus of the model. Applying this estimate on the corresponding subannulus with radii raised to the $1/p$ power extends the substitution continuously to every analytic $f$. It proves both analyticity and \eqref{eq:liftisometry} on the whole collar.
\end{proof}

\begin{corollary}[Compatible lifts through the tower]\label{cor:towerfrob}
Let $B\ge\max\{1,m_1,\ldots,m_\nu\}$ and choose $S<1$ with $\varepsilon_*=|\pi|S^{-B}<1$. Put $s_i=S^{1/p^{\nu-i}}$. Starting from $\sigma_\nu(w_\nu)=w_\nu^p$, the lifts of \cref{lem:onestepfrob} give compatible analytic maps on $A_{w_i}[s_i^{1/p},1)$. If
\[
 D_\nu=0,\qquad
 D_i=pD_{i+1}+p(p-1)m_{i+1},
\]
then
\begin{equation}\label{eq:liftcloseness}
 D_i\le pB(p^{\nu-i}-1),\qquad
 \left|\frac{\sigma_i(w_i)}{w_i^p}-1\right|_\varrho
       \le |\pi|\varrho^{-D_i}\le\varepsilon_*.
\end{equation}
For every $a\ge1$, the iterates $\phi_i=\sigma_i^a$, $q=p^a$, satisfy
\[
 \phi_i(w_i)=w_i^q(1+Z_i),\qquad
 |Z_i|_\varrho\le\varepsilon_*
 \quad\text{for }\varrho\ge s_i^{1/q}.
\]
The same radius $S$ works for every $a$.
\end{corollary}
\begin{proof}
The stated bound for $D_i$ follows by downward induction. At the step from $i+1$ to $i$,
\[
 pm_{i+1}+D_{i+1}\le Bp^{\nu-i},\qquad
 s_{i+1}^{1/p}=S^{1/p^{\nu-i}},
\]
so the Hensel residual in \cref{lem:onestepfrob} is at most $\varepsilon_*<1$. The same bound follows for the relative perturbation on $A_{w_i}[s_i^{1/p},1)$. Thus that lemma applies successively and gives compatible isometric lifts.

Write $\epsilon_i=\sigma_i(w_i)/w_i^p-1$. Direct iteration gives
\[
 \frac{\sigma_i^a(w_i)}{w_i^{p^a}}
   =\prod_{j=0}^{a-1}
       (1+\sigma_i^j(\epsilon_i))^{p^{a-1-j}}.
\]
For $\varrho\ge s_i^{1/p^a}$, the isometry gives
$|\sigma_i^j(\epsilon_i)|_\varrho
=|\epsilon_i|_{\varrho^{p^j}}\le\varepsilon_*$.
Integer powers and finite products preserve this perturbation bound, proving the assertion for every iterate. When $\nu=0$, the tower is empty and the lift is standard.
\end{proof}

\begin{lemma}[Rank-one change of lift]\label{lem:rankonechangelift}
Let $B\ge1$ and let a purely wild rank-one normal form on $A_w[s,1)$ have connection coefficient
\[
 \omega=\sum_{1\le k\le B}c_kw^{-k},\qquad
 |c_k|\le\kappa=|\pi|^{1/p^{j_*(B)}}.
\]
Consider two analytic lifts on $A_w[s^{1/q},1)$ with the same isometric coefficient action $\tau_K$,
\[
 \phi_{\mathrm{std}}(w)=w^q,\qquad
 \phi(w)=w^q(1+Z),\qquad |Z|_\varrho\le\varepsilon<1.
\]
Put $\mathscr H(x)=\sup_{n\ge1}x^n/|n|$ for $0\le x<1$. If
\[
 \kappa s^{-B}\mathscr H(\varepsilon)<|\pi|,
\]
there is a horizontal isomorphism $\phi^*\LL\to\phi_{\mathrm{std}}^*\LL$ whose factor and inverse are analytic of norm one on this collar.
\end{lemma}
\begin{proof}
Use the finite Laurent primitive
\[
 \Omega(z)=\sum_{1\le k\le B}\frac{c_k}{-k}z^{-k},
 \qquad \del_z\Omega=\omega.
\]
Its difference at the two lifts is
\[
 I=\Omega^{\tau_K}(\phi(w))-\Omega^{\tau_K}(w^q)
   =\sum_{1\le k\le B}c_k^{\tau_K}w^{-qk}f_k(Z),
 \qquad f_k(Z)=\frac{(1+Z)^{-k}-1}{-k}.
\]
The binomial identity
\[
 f_k(Z)=\sum_{n\ge1}\binom{-k-1}{n-1}\frac{Z^n}{n}
\]
has integral binomial coefficients. Consequently
\[
 |I|_\varrho\le|\omega|_{\varrho^q}\mathscr H(|Z|_\varrho)
       \le\kappa s^{-B}\mathscr H(\varepsilon)<|\pi|.
\]
Thus $\exp(I)$ and $\exp(-I)$ converge analytically and have norm one. The chain rule gives
\[
 \del_w I=
 \frac{\del_w\phi(w)}{\phi(w)}\omega^{\tau_K}(\phi(w))
       -q\omega^{\tau_K}(w^q).
\]
With our convention $\nabla_{\del_w}e=\omega e$, this says precisely that
$\phi^*e\mapsto\exp(I)\phi_{\mathrm{std}}^*e$ is horizontal.
\end{proof}

The function $\mathscr H$ has a simple explicit form in the range needed here:
\begin{equation}\label{eq:Hrange}
 \mathscr H(x)=\sup_{j\ge0}p^j x^{p^j}=p x^p
 \qquad\bigl(|\pi|\le x\le|\pi|^{1/p}\bigr).
\end{equation}
For the first equality write $n=p^j u$ with $p\nmid u$; among such $n$, the maximum occurs at $u=1$. For the second, compare adjacent terms: the term $j=1$ dominates $j=0$ when $x\ge|\pi|$, and dominates all subsequent terms when $x\le|\pi|^{1/p}$.

\begin{proposition}[Uniform Frobenius collar for wild pushforward]\label{prop:wildfrob}
Fix $c\in[|\pi|^p,1)$. Let $\LL$ be a purely wild rank-one normal form with exponent zero at the top of the tower in \cref{subsec:wildsetup}, and put
\[
 B=\max\{1,b(\LL),m_1,\ldots,m_\nu\}.
\]
Suppose a chosen Frobenius structure for $w_0\mapsto w_0^q$, $q=p^a$, is analytic with analytic inverse on the Frobenius subannulus of a model with inner radius $s_{\mathrm{in}}$. There is a bottom radius $S<1$, depending only on $B,p,c,s_{\mathrm{in}}$, such that this structure pushes forward through the tower relative to the standard bottom lift $w_\nu\mapsto w_\nu^q$, with analytic inverse on $A[S^{1/q},1)$. The radius $S$ may be chosen at least as large as all connection-model radii in \cref{prop:wild}, independently of $a$. The construction is functorial in the chosen initial structure.
\end{proposition}
\begin{proof}
Put $\kappa=|\pi|^{1/p^{j_*(B)}}$ and choose
\begin{equation}\label{eq:taylorcollar}
 \max\left\{s_{\mathrm{in}},c^{1/(B+1)},
             \kappa^{1/((p+1)B)}\right\}<S<1.
\end{equation}
The recurrence in \cref{prop:wild} gives $B_i\le B$ at every stage. Thus $S$ is at least all the connection-model radii $c^{1/(B_i+1)}$.

Set $s_i=S^{1/p^{\nu-i}}$ and $\varepsilon_*=|\pi|S^{-B}$. Since $S^{(p+1)B}>\kappa\ge|\pi|$,
\[
 |\pi|<\varepsilon_*<|\pi|\kappa^{-1/(p+1)}
       \le|\pi|^{p/(p+1)}<|\pi|^{1/p}.
\]
\Cref{eq:Hrange} therefore gives
\[
 \kappa S^{-B}\mathscr H(\varepsilon_*)
       =\kappa|\pi|S^{-(p+1)B}<|\pi|.
\]
Apply \cref{cor:towerfrob} and then \cref{lem:rankonechangelift} at the top, with $s=s_0\ge S$. The latter transports the chosen standard-lift Frobenius to the compatible lift $\phi_0$, with analytic inverse.

The compatible lifts give analytic comparison maps for the squares at every wild step. These squares are cartesian over the Robba ring by \cref{rem:compatibility}. Their pullback and upper section algebras are finite free \'etale of the same rank, so \cref{cor:cartesiancollar} makes the squares cartesian on the whole collar. Apply \cref{lem:frobpush} successively. The bottom lift is standard, and restriction to these collars preserves the connection estimates. Both change of lift and finite \'etale pushforward are functorial.
\end{proof}

For the constituent construction there are exactly $\nu=v_p(d)$ wild steps, and \cref{prop:geometric} bounds $B$ by $\max\{1,B_*(d,b,p)\}$. After fixing $c=c(p)$ in the final assembly, the analytic collar and the pushforward constants therefore have the required dependence on the constituent rank and ramification data.

\section{Dwork's trick and differential direct summands}\label{sec:dwork}
\subsection{Extension of horizontal endomorphisms}
Dwork's argument extends horizontal projectors to the ambient model annulus. We give the annular continuation argument explicitly to track the inner radius; for the classical disc form of Dwork's trick, see \cite[Section~17.2]{Ked22}. The same argument will then extend the prescribed Frobenius on their images.
Let $E$ be a finite free differential module of rank $r_E$ over $\mathscr A_s$, and put $E_{\mathcal{R}}=E\otimes_{\mathscr A_s}\mathcal{R}_{K,t}$. For \cref{lem:endext}, assume $E$ has a horizontal Frobenius isomorphism for $\phi(t)=t^q$, where $q>1$ and the coefficient lift is isometric. In a fixed basis write its matrix as $\FM $, with
\[
 \FM ,\FM ^{-1}\in M_{r_E}(\mathscr A_{s^{1/q}}).
\]

\begin{lemma}\label{lem:unique}
If an $\mathscr A_s$-direct summand $\mathcal{M}\subseteq E$ has Robba-ring extension $M\subseteq E_{\mathcal{R}}$, then $\mathcal{M}=E\cap M$ inside $E_{\mathcal{R}}$.
\end{lemma}
\begin{proof}
Choose $E=\mathcal{M}\oplus\mathcal{N}$. For $e\in E\cap M$, write $e=m+n$. After base change, $n$ lies both in $M$ and in $\mathcal{N}_{\mathcal{R}}$, so $n=0$. The reverse inclusion is immediate.
\end{proof}

\begin{lemma}[Dwork extension]\label{lem:endext}
Every horizontal endomorphism of $E_{\mathcal{R}}$ has matrix entries in $\mathscr A_s$.
\end{lemma}
\begin{proof}
The space $V=H^0_\nabla(\End(E_{\mathcal{R}}))$ is a finite-dimensional $K$-vector space, of dimension at most $r_E^2$. Frobenius induces the semilinear map
\[
 \mathscr F(Q)=\FM \phi(Q)\FM ^{-1}
\]
on $V$. Put $d=\dim_K V$, choose a $K$-basis $Q_1,\ldots,Q_{d}$, and write $\mathscr F(Q_j)=\sum_i a_{ij}Q_i$ with $a_{ij}\in K$. Thus
\begin{equation}\label{eq:bootstrap}
 \phi(Q_j)=\FM ^{-1}\left(\sum_i a_{ij}Q_i\right)\FM .
\end{equation}
Initially all $Q_i$ are analytic on some common $A[\rho_0,1)$ with $s\le\rho_0<1$. The right side of \eqref{eq:bootstrap} is analytic on $A[\max\{\rho_0,s^{1/q}\},1)$. For a Laurent series $Q_j=\sum_n Q_{j,n}t^n$, isometry of the coefficient lift implies that convergence of $\phi(Q_j)$ at a radius $\rho$ is equivalent to convergence of $Q_j$ at $\rho^q$. Therefore all $Q_j$ are analytic on
\[
 A[\max\{\rho_0^q,s\},1).
\]
Repeating gives analyticity on $A[\max\{\rho_0^{q^j},s\},1)$. Since $\rho_0<1$ and $s>0$, finitely many iterations reach $s$.
\end{proof}

\begin{corollary}[Horizontal maps and coefficient extension]\label{cor:homextension}
If finite free differential models $E_1,E_2$ on $A_{K,t}[s,1)$ carry auxiliary Frobenius structures for the same lift and level, every horizontal Robba-ring map $E_{1,\mathcal{R}}\to E_{2,\mathcal{R}}$ extends to the whole model annulus. The same conclusion holds after any finite coefficient extension, even when the auxiliary coefficient Frobenius does not extend to that field.
\end{corollary}
\begin{proof}
Apply \cref{lem:endext} to $E_1\oplus E_2$ and take the appropriate off-diagonal block. For a finite extension $L/K$, expand the entries of a horizontal map in a $K$-basis of $L$. Since the connection matrices have coefficients in $K$, every coefficient matrix is horizontal over $K$ and hence extends. The finite expansion then extends over $L$.
\end{proof}

\begin{proposition}[Extension of differential direct summands]\label{prop:extend}
Let $E$ be a finite free differential module over $\mathscr A_s$ such that every horizontal endomorphism of $E_{\mathcal{R}}$ extends to $E$. Then every differential direct summand $M\subseteq E_{\mathcal{R}}$ extends uniquely to an $\mathscr A_s$-differential direct summand $\mathcal{M}\subseteq E$.
\end{proposition}
\begin{proof}
Choose a horizontal projector ${e_M}$ onto $M$. By hypothesis its entries lie in $\mathscr A_s$, where the identities
\[
 {e_M}^2={e_M},\qquad \del_t({e_M})+A{e_M}-{e_M} A=0
\]
hold by injectivity of $\mathscr A_s\to\mathcal{R}_{K,t}$. Thus $E=\operatorname{im}{e_M}\oplus\ker{e_M}$ is a differential direct-sum decomposition. Its image $\mathcal{M}=\operatorname{im}{e_M}$ is finite projective and hence defines a model of $M$. Uniqueness follows from \cref{lem:unique}.
\end{proof}

The extension hypothesis holds when $E$ carries an auxiliary Frobenius, by \cref{lem:endext}, and also for the descended ambient models of \cref{prop:ambientdescent}. The estimates on its image come from the induced norm at each radius.

\begin{proposition}[Extension of a prescribed Frobenius]\label{prop:prescribedextension}
Suppose a finite free differential model $E$ on $A_{K,t}[s,1)$ carries an auxiliary Frobenius structure for $\sigma^{m}$, for some $m\ge1$. Let $M\subset E_{\mathcal{R}}$ be the image of a horizontal projector ${e_M} $, and let $\mathcal{M}=\operatorname{im} {e_M} \subset E$ be its extended model. Every prescribed horizontal isomorphism
\[
 \Phi:(\sigma^j)^*M\xrightarrow{\sim}M
\]
extends to an isomorphism $(\sigma^j)^*\mathcal{M}\xrightarrow{\sim}\mathcal{M}$ on $A_{K,t}[s^{1/p^j},1)$. The projector need not be equivariant for the auxiliary Frobenius.
\end{proposition}
\begin{proof}
Work on $A[s^{1/p^j},1)$. Both $E$ and $(\sigma^j)^*E$ carry auxiliary $\sigma^{m}$-Frobenius structures, since powers of $\sigma$ commute. If $\iota:M\hookrightarrow E_{\mathcal{R}}$ and $\operatorname{pr}:E_{\mathcal{R}}\to M$ are the inclusion and projection, form the horizontal maps
\[
 F=\iota\Phi\,(\sigma^j)^*\operatorname{pr}:(\sigma^j)^*E_{\mathcal{R}}\to E_{\mathcal{R}},
 \qquad
 G=(\sigma^j)^*\iota\,\Phi^{-1}\operatorname{pr}:E_{\mathcal{R}}\to(\sigma^j)^*E_{\mathcal{R}}.
\]
By \cref{cor:homextension}, they extend to this annulus. The Robba-ring identities
\[
 {e_M} F=F=F(\sigma^j)^*{e_M} ,\quad (\sigma^j)^*{e_M} G=G=G{e_M} ,
 \qquad FG={e_M} ,\quad GF=(\sigma^j)^*{e_M} 
\]
hold on the annulus by injectivity of restriction to the Robba ring. Restricting $F,G$ to the two images gives the required inverse isomorphisms.
\end{proof}

\subsection{Orthonormal bases on a circle}
Write $\Gamma^*=|K^\times|^{\Q}$ for the divisible closure of the value group, and equip $R_\rho^{r_E}$ with its coordinate sup norm. The cases $\rho\notin\Gamma^*$ and $\rho\in\Gamma^*$ treated below correspond to type~3 and type~2 Gauss points, respectively \cite[Section~1.4]{Ber90}. We show that, at every radius, every submodule of $R_\rho^{r_E}$ has an orthonormal basis for its induced norm. This will transfer the ambient connection bound to a differential summand. The proof uses the standard maximal-minor and reduction arguments, which we include for completeness.

\begin{proposition}[Orthonormal bases on circles]\label{prop:orthonormal}
For every $\rho\in(0,1)$, every $R_\rho$-submodule $M\subseteq R_\rho^{r_E}$ admits an orthonormal basis for its induced norm.
\end{proposition}
\begin{proof}
The zero module is immediate. We treat the two kinds of radius separately.

\smallskip\noindent\emph{Case 1: $\rho\notin\Gamma^*$.}
For a nonzero series $f=\sum_n a_nt^n$, two distinct monomials cannot have equal norm: equality would imply $\rho^{n-m}\in|K^\times|$. Since the monomial norms tend to zero in both directions, one strictly dominates the others. Thus $f=a_nt^n(1+h)$ with $|h|_\rho<1$, and the geometric series gives its inverse. Hence $R_\rho$ is a field.

Let the columns of an $r_E\times d$ matrix $B$ form a basis of $M$. Choose $d$ row indices $I$ for which $|\det B_I|_\rho$ is maximal, and set $F=BB_I^{-1}$. Then $F_I=I_d$, and Cramer's rule gives $|F_{ij}|_\rho\le1$. For every $x\in R_\rho^d$,
\[
 |x|_\rho\le|Fx|_\rho\le|x|_\rho.
\]
The columns of $F$ are therefore orthonormal. This is the usual subspace property for orthonormal bases over a valued field; compare \cite[Section~2.5]{BGR84}.

\smallskip\noindent\emph{Case 2: $\rho\in\Gamma^*$.}
Choose a uniformizer $\varpi$ of $K$ and write $\rho=|\varpi|^{a/e}$ with $\gcd(a,e)=1$. Choose integers $\alpha,\beta$ with $e\alpha+a\beta=1$, and put
\[
 w=\varpi^\alpha t^\beta,\qquad z=\varpi^{-a}t^e,\qquad
 S_\rho=R_\rho^\circ,\qquad \Lambda=M\cap S_\rho^{r_E}.
\]
The value $|w|_\rho=|\varpi|^{1/e}$ generates the norm group, so $R_\rho^{\circ\circ}=wS_\rho$. The canonical reduction is
\[
 S_\rho/wS_\rho\simeq k[\bar z,\bar z^{-1}].
\]
Indeed, the norm-one monomials reduce to powers of $z$, and all monomials of norm below one lie in $wS_\rho$. For the general reduction construction, see \cite[Section~2.4]{Ber90}.

Since $w$ is invertible in $R_\rho$, $\Lambda\cap wS_\rho^{r_E}=w\Lambda$. Thus $\Lambda/w\Lambda$ embeds in $k[\bar z,\bar z^{-1}]^{r_E}$ and is finite free over the Laurent polynomial PID. Lift a basis to $f_1,\ldots,f_d\in\Lambda$. Successive approximation modulo powers of $w$ expresses every $x\in\Lambda$ as a linear combination of the $f_i$ with coefficients in $S_\rho$: at step $j$ the error lies in $w^j\Lambda$, and the coefficient series converge in the complete ring $S_\rho$.

If $a_i\in S_\rho$ and $\max_i|a_i|_\rho=1$, independence of the reduced basis implies $|\sum_i a_if_i|_\rho=1$. Rescaling by a power of $w$ gives
\[
 \left|\sum_i a_if_i\right|_\rho=\max_i|a_i|_\rho
 \qquad(a_i\in R_\rho).
\]
Finally $\Lambda[1/w]=M$, so the $f_i$ form an orthonormal $R_\rho$-basis of $M$.
\end{proof}

\subsection{Connection bounds on the image}
The ambient connection has a bounded operator norm. An orthonormal basis transfers this bound to the differential subbundle.
\begin{corollary}\label{cor:connectionimage}
Let $\mathcal{E}$ witness $(s,N)$-boundedness and let $\mathcal{M}\subseteq\mathcal{E}$ be a differential subbundle. Then $\mathcal{M}$ is $(s,\max\{N,1\})$-bounded.
\end{corollary}
\begin{proof}
Fix $\rho$, and write $D=\nabla_{\del_t}=\del_t+A_\rho$ in the chosen ambient basis. For every Laurent series, $|\del_t f|_\rho\le|f|_\rho$, so
\[
 \norm{Dv}_\rho\le\max\{N,1\}\norm v_\rho.
\]
Choose an orthonormal basis $f_1,\ldots,f_d$ of $\mathcal{M}_\rho$ by \cref{prop:orthonormal}. If $Df_j=\sum_i b_{ij}f_i$, then
\[
 \max_i|b_{ij}|_\rho=\norm{Df_j}_\rho
       \le\max\{N,1\}\norm{f_j}_\rho=\max\{N,1\}.
\]
This gives the asserted connection bound at every radius.
\end{proof}

\begin{remark}[Scope of the radiuswise bounds]\label{rem:globalbasis}
The explicit ambient model, including after coefficient restriction, has a global-basis witness. Dwork's trick extends the projector onto the desired summand qualitatively, without a uniform estimate for its norm. Its columns generate the image, but even a bound on those generators would not by itself supply a global basis with controlled connection matrix: that requires a quantitative trivialization, which qualitative freeness alone does not provide. The proof of \cref{cor:connectionimage} instead uses the induced ambient norm and chooses an orthonormal basis separately on each circle. It needs no projector-norm estimate, but establishes no analytic compatibility between these bases as the radius varies.

These are limitations of the present argument, not evidence that a uniformly bounded global basis is impossible. In particular, the argument leaves open whether every $M\in\QU(r,b)$ admits a global-basis witness with inner radius and connection bound depending only on $r,b,p$, either with the constants of \cref{thm:A} or after changing those constants. Such a bound would be an additional assertion; compatibility between models on different test curves would still require a separate argument.
\end{remark}

\begin{remark}[Radiuswise Frobenius normalization]\label{rem:radiusmeaning}
Let $\mathcal{E}$ be a $(\sigma^m,\nabla)$-model on $A[s,1)$ with a radiuswise connection bound $N$, and put $q=p^m$. Choose witnessing bases on the radii in $[s,s^{1/q})$, and transport them outwards by Frobenius. On the interval $[s^{1/q^j},s^{1/q^{j+1}})$ the resulting connection matrices have norm at most $|q|^jN$, while the Frobenius matrices between the chosen bases are identities. This follows from horizontality: $A_\rho=q\sigma^m(A_{\rho^q})$, the same Frobenius-transport argument as in \cite[proof of Lemma~6.1.3]{Ked05}. It is a normalization of independently chosen circle bases, not an additional quantitative estimate.

For example, for a constant rank-one module over $\Q_p$ with $\Phi(e)=ce$, the bases $e_\rho=c^je$ on these intervals give identity Frobenius matrices for every $c\in\Q_p^\times$. Thus the normalization does not detect the Newton slope. It neither supplies the model annulus nor produces a single global witnessing basis. Estimates in a prescribed analytic basis would require control of the change of basis.
\end{remark}

\section{Assembly of the main theorems}\label{sec:assembly}
\subsection{Matsuda decomposition and the quasi-constant reduction}\label{subsec:matsuda}
The preceding sections construct and control models for the quasi-constant constituents. We now use Matsuda's theorem to separate finite local monodromy from unipotent extensions and assemble the models for a general quasi-unipotent module.
\begin{theorem}[Matsuda decomposition]\label{thm:matsuda}
Every quasi-unipotent differential module $M$ over $\mathcal{R}_{K,t}$ admits a decomposition
\[
 M\simeq\bigoplus_i \mathcal{F}_i\otimes_{\mathcal{R}_{K,t}}U_i,
\]
where the $\mathcal{F}_i$ are pairwise nonisomorphic irreducible quasi-constant modules and the $U_i$ are nonzero unipotent modules. The summand corresponding to $\mathcal{F}_i$ is the canonical $\mathcal{F}_i$-isotypical part of $M$: all of its irreducible differential subquotients are isomorphic to $\mathcal{F}_i$.
\end{theorem}
\begin{proof}
This is \cite[Theorem~4.6.6]{Ked05}, following \cite[Theorem~7.8]{Mat02}.
\end{proof}

The ramification estimates concern the quasi-constant factors $\mathcal{F}_i$, while the unipotent factors $U_i$ can be handled on every annulus. In particular,
\[
 \rk M=\sum_i\rk(\mathcal{F}_i)\rk(U_i),\qquad
 b(M)=\max_i b(\mathcal{F}_i).
\]
The theorem is over $K$ itself. The finite coefficient extension in the next lemma is used only to make the quasi-constant factors absolutely irreducible, so that the rank-one realization of \cref{sec:induction} applies.

\begin{lemma}[Matsuda reduction after coefficient extension]\label{lem:coeffsplitting}
For every $M\in\QU(r,b)$, there exists a finite coefficient extension $K'/K$, depending on $M$, such that
\[
 M_{K'}\simeq\bigoplus_i \mathcal{F}_i\otimes U_i,
\]
where the $\mathcal{F}_i$ are absolutely irreducible quasi-constant modules and the $U_i$ are unipotent. One has $b(\mathcal{F}_i)\le b$ and $\sum_i\rk(\mathcal{F}_i)\rk(U_i)\le r$.
\end{lemma}
\begin{proof}
Apply \cref{thm:matsuda} over $K$, and then apply \cref{lem:absolutesplitting} to its finitely many quasi-constant factors. Choose a common finite coefficient extension splitting all of them and distribute the tensor products over their direct-sum decompositions. Unipotent factors have constant Jordan models and remain unipotent after extension. Tensoring with a nonzero unipotent factor preserves the highest inertia break, which gives the asserted bounds.
\end{proof}

\subsection{Unipotent factors}
\begin{lemma}\label{lem:unipotent}
A unipotent differential module $U$ has $b(U)=0$ and is $(s,1)$-bounded with a global-basis witness for every $s\in(0,1)$.
\end{lemma}
\begin{proof}
The classification of unipotent differential modules gives a basis in which $\nabla_{\del_t}$ has a constant nilpotent matrix; see \cite[Proposition~4.5.3]{Ked05}. A constant change of basis puts this matrix into Jordan form, with entries zero or one. This defines a model on every annulus, with connection norm at most one. The inertia representation of the unipotent factor has no positive ramification break.
\end{proof}

\subsection{Elementary operations}
\begin{lemma}\label{lem:operations}
Suppose $M_i$ is $(s_i,N_i)$-bounded. Put $s=\max_i s_i$ and $N=\max_i N_i$. Then $\bigoplus_iM_i$ and $\bigotimes_iM_i$ are $(s,N)$-bounded. The same assertion holds for global-basis witnesses. Duals preserve the connection bound.
\end{lemma}
\begin{proof}
Restrict all models to $A_{K,t}[s,1)$. In the direct-sum basis, the connection matrix is block diagonal. In the tensor-product basis it is
\[
 \sum_i I\otimes\cdots\otimes A_i\otimes\cdots\otimes I.
\]
The ultrametric inequality gives norm at most $N$ in both cases. The dual connection matrix is $-A_i^{\mathsf T}$.
\end{proof}

\subsection{Realization in normal form}
\begin{proposition}[Realization in normal form]\label{prop:geometric}
For an absolutely irreducible constituent of rank $d$ and break at most $b$, after a finite coefficient extension the inducing data of \cref{prop:induced} admit tame power maps and the wild normal forms of \cref{subsec:rankonenormal,subsec:wildsetup}, with their equations and connections defined over $C_\ell F_{j_*}$ for some finite $\ell/k$, where $j_*=j_*(B_*(d,b,p))$.
The purely wild inducing factor $\widetilde\LL$ of \cref{subsec:tametwist} and the wild steps satisfy
\[
 b(\widetilde\LL)\le B_*(d,b,p),\qquad
 m_{i+1}\le B_*(d,b,p)\quad(0\le i<\nu).
\]
\end{proposition}
\begin{proof}
By \cref{rem:finiterealization}, the finite monodromy group and all intermediate subgroups in the inducing construction are realized by residue covers after a finite coefficient extension. Realize the finitely many character values and intertwining matrices over a further finite extension if necessary. We may enlarge the finite residue extension $\ell/k$ during the following choices.

\smallskip\noindent\emph{Coordinates and covering equations.}
In residue coordinates, a tame map of degree $a$ has the form $x=cv^aU(v)$, where $c\in\ell^\times$ and $U(v)\in1+v\ell[[v]]$. Adjoin an $a$th root of $c$. Since $(a,p)=1$, Hensel's lemma gives an $a$th root of $U(v)$ in $1+v\ell[[v]]$, and the new coordinate $c^{1/a}vU(v)^{1/a}$ makes the map a power map. Starting with the fixed coordinate $t$, do this first for $h_2$ and then for $h_1$, obtaining $u^{d'}=t$ and $w_\nu^{n_1}=u$ in the notation of \cref{fig:cover-tower}.

Proceed upwards through the wild tower, from $w_\nu$ to $w_0$. At each step choose a reduced Artin--Schreier equation with a negative polynomial part and a constant part. A finite extension of $\ell$ removes the constant part. The uniformizer $y_i^{-\alpha_{i+1}}w_{i+1}^{-\beta_{i+1}}$ then supplies the next coordinate, as in \cref{subsec:wildsetup}. The same finite-extension procedure removes the constant Witt part of the inducing rank-one character. All these choices require only finitely many algebraic residue elements; unramified coefficient extensions realize them, and the geometric degrees and breaks are unchanged.

Choose Teichm\"uller lifts of the reduced covering polynomials, and lift the wild equations by \eqref{eq:wildequation}. These equations have the same finite \'etale residue covers as the original inducing data. The henselian lifting correspondence for the integral Robba ring identifies the two lifts \cite[Proposition~3.5.5 and Definition~3.5.6]{Ked05}; their isomorphisms and inverses are therefore analytic on a sufficiently small outer collar. This identification is needed only over the Robba ring; the uniform model radii are supplied by the explicit estimates of \cref{sec:push}.

\smallskip\noindent\emph{The inducing character and its bounds.}
Remove the scalar tame factor as in \cref{lem:twist}. Fix the identification of cyclic $p$-power character groups with roots of unity associated with the chosen Lubin--Tate parameters \cite[Definition~4.14]{Pul07}. Choose the Artin--Schreier--Witt class whose associated connection realizes exactly $\lambda\boxtimes1$, using the character--connection correspondence of \cite[Section~6.3]{CP09}. With our convention $\omega=\del_w\log Y$, the horizontal generator is $e/Y$, so its character is the inverse of the character on $Y$. If an initial choice gives a unit power of the desired character, multiply its Witt class by an inverse unit modulo its $p$-power order, represented by an integer prime to $p$, and then reduce again. This multiplication is in the Witt group, not componentwise multiplication of its coordinates. Thus the normal form realizes the required character, including its sign convention.

Write $n_1=[T:T_0]$. Every subgroup occurring between $Q\times T_0$ and $I$ has index at most $dn_1\le d(p^{2d}-1)$. By \cref{lem:subgroupbreak}, both the inducing rank-one character and the cyclic degree-$p$ quotients defining the wild steps have break at most $B_*(d,b,p)$. Removing a tame factor preserves all positive breaks, proving the displayed bounds. By \eqref{eq:rankonebreak}, every Lubin--Tate level $j$ needed for the reduced purely wild character satisfies $p^j\le B_*(d,b,p)$; if this bound is less than one, the purely wild character is trivial. Hence only $\pi_0,\ldots,\pi_{j_*}$ are needed. Use the Witt-monomial lift of \cref{subsec:rankonenormal}, and adjoin $F_{j_*}$.

\smallskip\noindent\emph{Field of definition.}
Take $\ell$ large enough to contain the residue coefficients and the constants used in these normalizations. A tame differential factor has an exponent in $\mathbb{Z}_{(p)}$; extending its scalar character across a tame cyclic quotient divides the exponent by a prime-to-$p$ integer. The following table records sufficient coefficient fields for the chosen equations and connections:
\begin{center}
\small
\setlength{\tabcolsep}{4pt}
\renewcommand{\arraystretch}{1.2}
\begin{tabular}{@{}p{0.23\textwidth}p{0.44\textwidth}p{0.25\textwidth}@{}}
\toprule
Object & Coefficients & Field \\
\midrule
Tame covers & Rational coefficients in $u^{d'}=t$, $w_\nu^{\,n_1}=u$ & $\Q_p$ \\
Wild covers & Teichm\"uller coefficients of $P_{i+1}$, and $\pi_0$ & $C_\ell F_0\subseteq C_\ell F_{j_*}$ \\
Rank-one $\widetilde\LL$ & Teichm\"uller Witt-monomial coefficients, and $\pi_0,\ldots,\pi_{j_*}$ & $C_\ell F_{j_*}$ \\
Tame twists & Exponents in $\mathbb{Z}_{(p)}$ & $\Q_p$ \\
\bottomrule
\end{tabular}
\end{center}
The comparison with the original constituent takes place over a common finite coefficient extension containing $K$ and $C_\ell F_{j_*}$. The auxiliary coefficient Frobenius preserves $C_\ell$ and fixes $F_{j_*}$, by \cref{rem:coefficientlift}.
\end{proof}

\subsection{The explicit ambient model}
\begin{proposition}\label{prop:ambientmodel}
There are $s=s(r,b,p)<1$ and $N_0=N_0(r,b,p)\ge1$ with the following property. Put $j_*=j_*(B_*(r,b,p))$. For $M\in\QU(r,b)$ over $K$, there is a finite extension $L/K$, with residue field $\ell$, and an explicit differential model $V_0$ over
\[
 F=C_\ell F_{j_*},\qquad
 C_\ell=\operatorname{Frac}W(\ell),\qquad F_{j_*}=\Q_p(\pi_{j_*}),
\]
such that:
\begin{enumerate}
\item $F\subseteq L$, and $L/F$ is finite;
\item $V_0$ has a global-basis $(s,N_0)$-witness, and $M_L$ is a differential direct summand of the differential Robba module associated with its base change to $L$;
\item $V_0$ carries an auxiliary Frobenius structure on this model annulus for some power of the lift acting by Witt Frobenius on $C_\ell$, fixing $F_{j_*}$, and sending $t$ to $t^p$.
\end{enumerate}
\end{proposition}
\begin{proof}
\noindent\emph{Inducing data.}
Use \cref{lem:coeffsplitting}. For an absolutely irreducible constituent of rank $d$ and break at most $b$, \cref{prop:geometric} realizes the inducing data after a finite coefficient extension. Write $\nu=v_p(d)$ and $n_1=[T:T_0]$. The maps in \cref{prop:induced} have degrees
\[
 \deg g=p^\nu,\qquad \deg h_1=n_1\le p^{2d}-1,
 \qquad \deg h_2=d/p^\nu.
\]
By \cref{prop:geometric}, the inducing rank-one character and every wild step have break at most $B_*=B_*(d,b,p)$.

\smallskip\noindent\emph{Analytic models.}
Use the differential realization \eqref{eq:ambient}: with $M_\psi$ and $\mathcal T_1$ defined in \cref{subsec:tametwist}, the constituent is a differential summand of $h_{2*}(V\otimes\mathcal T_1)$.
Fix $c=c(p)\in(|\pi|^p,1)$. \Cref{lem:rankone} and \cref{prop:wild} give a global-basis connection witness with
\[
 N_\nu\le(p/c)^\nu c^{-1}.
\]
For the model radius, start with the rank-one Frobenius collar of \cref{lem:rankonefrob} and choose the bottom wild radius $S$ supplied by \cref{prop:wildfrob}, with $B=\max\{1,B_*\}$. By its construction, $S$ is at least all connection-model radii. Pushing through the two tame maps gives the base model radius
\[
 s_d=S^{\deg h_1\deg h_2}.
\]
Both the connection and the auxiliary Frobenius are therefore defined on this same model annulus, with Frobenius on its usual outer subannulus. The tame connection estimate in \cref{lem:tame} depends only on these degrees and radii. Tensoring with the tame factors, whose constant connection coefficients lie in $\mathbb{Z}_{(p)}$, does not increase a bound already at least one.

\smallskip\noindent\emph{Coefficient fields and auxiliary Frobenius.}
All equations are defined over $F$: the covering coefficients and rank-one Witt lifts lie in $W(\ell)$, the required parameters lie in $F_{j_*}$, and tame exponents are rational. The coefficient lift of \cref{rem:coefficientlift} gives the purely wild auxiliary Frobenius. Use \cref{rem:elementaryfrob} for the tame and unipotent factors, and \cref{rem:tamefrob} to carry Frobenius through the tame maps. Passing to a common level, tensoring, and taking direct sums preserve the model annulus.

\smallskip\noindent\emph{Assembly.}
For each factor $\mathcal{F}_i\otimes U_i$ in \cref{lem:coeffsplitting}, tensor the ambient model for $\mathcal{F}_i$ with the global-basis model for $U_i$ of \cref{lem:unipotent}. Combine these models by direct sums using \cref{lem:operations}. Use the common cutoff $j_*=j_*(B_*(r,b,p))$, take a common finite coefficient extension, and take the maximum of the finitely many base radii and connection bounds. The number and ranks of the constituents are at most $r$, so nondecreasing envelopes give $s,N_0$ depending only on $r,b,p$. Finally, $L/F$ is finite because $K/C_k$ and $L/K$ are finite and $C_k\subset F\subset L$.
\end{proof}

\subsection{An ambient model over the original field}\label{subsec:descent}
Restriction of scalars places the explicit models over a common Witt coefficient field. The following estimate makes its cost independent of the coefficient-extension degree.
\begin{lemma}[Coefficient restriction]\label{lem:coefficientrestriction}
Let $F/C$ be a finite extension of complete discretely valued coefficient fields with perfect residue fields. If $V$ over $A_{F,t}[s,1)$ has a global-basis $(s,N_0)$-witness, then $\operatorname{Res}_{F/C}V$ has a global-basis $(s,pN_0)$-witness. If $V$ carries an auxiliary Frobenius for an isometric lift on $F$ extending one on $C$, restriction of scalars carries that structure and its analytic inverse to the same model annulus over $C$.
\end{lemma}
\begin{proof}
Choose an orthogonal $C$-basis $b_1,\ldots,b_f$ of $F$, rescaled so that $|p|<|b_j|\le1$. Such a basis is obtained from residue-basis lifts and powers of a uniformizer, followed by rescaling over $C$; for background on orthogonal bases, see \cite[Chapter~2]{BGR84}. At every radius,
\[
 \left|\sum_j b_j a_j(t)\right|_\rho
       =\max_j|b_j|\,|a_j|_\rho,
       \qquad a_j\in R_{C,\rho}.
\]
Thus the weighted norm and the coordinate sup norm differ by at most a factor $p$. Multiplication by a matrix of norm at most $N_0$ consequently has coordinate norm at most $pN_0$. In the basis $b_je_l$ obtained from a global witnessing basis of $V$, the $b_j$ have zero derivative, proving the connection estimate.

The coefficient lifts are automorphisms, as observed before \cref{thm:B}. The coefficient Frobenius therefore gives an invertible constant change between the bases $(b_j)$ and their images. Together with the given Frobenius of $V$, this induces an analytic Frobenius isomorphism and inverse on the restriction of scalars.
\end{proof}

We now restrict scalars to $C_k$ and extend scalars to $K$. This preserves uniform control of the connection while retaining the horizontal-extension property.
\begin{proposition}\label{prop:ambientdescent}
With the notation of \cref{prop:ambientmodel}, the module $M$ is a differential direct summand of the Robba module associated with a global-basis $(s,pN_0)$-bounded model $E$ over $K$. Every horizontal endomorphism of $E_{\mathcal{R}}$ extends to $E$. If $K$ carries an isometric absolute lift $\sigma_K$, then $E$ also carries an auxiliary structure for a power of $\sigma$.
\end{proposition}
\begin{proof}
Put $C=C_k$ and
\[
 \mathcal{W}=\operatorname{Res}_{F/C}V_0,\qquad E_0=\mathcal{W}\otimes_C K.
\]
The extension $F/C$ is finite. By \cref{lem:coefficientrestriction}, $\mathcal{W}$ has a global-basis connection bound $pN_0$ and an auxiliary Frobenius for a power of Witt Frobenius on $C$. Base change preserves its matrix norms, so $E_0$ has the same connection bound.

The multiplication map $L\otimes_C F\to L$ exhibits $L$ as a direct factor of this finite \'etale $L$-algebra. Hence $V_{0,L}$ is a differential direct summand of $E_{0,L}$, and so is $M_L$. Set $d_L=[L:K]$. Restriction of scalars gives a differential splitting inside
\[
 \operatorname{Res}_{L/K}E_{0,L}\simeq E_0^{\oplus d_L}.
\]
The normalized trace $\operatorname{Tr}_{L/K}/d_L$ splits $M$ off $\operatorname{Res}_{L/K}M_L$. Thus $M$ is a differential direct summand of $E_{\mathcal{R}}$, where $E=E_0^{\oplus d_L}$. Direct sums preserve the bound $pN_0$.

Apply \cref{cor:homextension} to $\mathcal{W}^{\oplus d_L}$ over $C$, followed by the finite extension $K/C$. Every horizontal endomorphism of $E_{\mathcal{R}}$ extends to $E$, even if $K$ has no Frobenius lift. Neither this argument nor the connection bound involves $[K:C]$ or the norm of the trace projector.

If $\sigma_K$ is prescribed, \cref{rem:coefficientlift} gives $\sigma_K|_C=\sigma_C$. The auxiliary Frobenius on $\mathcal{W}^{\oplus d_L}$ therefore base changes to an auxiliary $\sigma^{m}$-structure on $E$ for some $m\ge1$.
\end{proof}

\subsection{Proofs of the main theorems}\label{subsec:mainproofs}
For \cref{thm:A}, choose $E$ as in \cref{prop:ambientdescent} and a horizontal Robba-ring projector ${e_M} $ onto $M$. The horizontal-extension property in \cref{prop:ambientdescent} satisfies the hypothesis of \cref{prop:extend}. Thus ${e_M}$ extends and its image is a differential model of $M$ over $K$. By \cref{cor:connectionimage}, it is $(s,N)$-bounded with
\[
 N(r,b,p)=\max\{1,pN_0(r,b,p)\}.
\]

For \cref{thm:B}, apply the construction above to $(M,\Phi)$, with $s=s(r,b,p)$ and $N=N(r,b,p)$. By \cref{prop:ambientdescent}, its ambient model carries an auxiliary Frobenius for a power of the prescribed lift. \Cref{prop:prescribedextension}, with $j=1$, extends the original $\Phi$ and its inverse to the extracted model $\mathcal{E}$ on $A_{K,t}[s^{1/p},1)$. Thus $\mathcal{E}$ is a $(\sigma,\nabla)$-model on $A_{K,t}[s,1)$ with the same radiuswise connection bound $N$.

\subsection{Dependence of the constants}\label{subsec:constants}
For \cref{thm:A}, the proof provides a finite recipe for the constants. Set
\[
 B_* = B_*(r,b,p),\qquad B=\max\{1,B_*\},
 \qquad \nu_* = \lfloor\log_p r\rfloor,
\]
and choose $c\in(|\pi|^p,1)$ as a function of $p$. Choose a common elementary-exponential radius $R_\Theta(B,p)>1$ as in \cref{lem:rankonefrob}; this is the only analytic input from the fixed Lubin--Tate exponentials. The initial Frobenius radius, the connection-radius bound $c^{1/(B+1)}$, and the explicit condition $S^{(p+1)B}>\kappa$ in \eqref{eq:taylorcollar} determine a permissible wild collar $S$. Since $c=c(p)$ has now been fixed, all these parameters depend only on $B,p$. For each possible pair of tame degrees, push this radius down by their product and apply \cref{lem:tame} twice to the starting bound
\[
 (p/c)^{\nu_*}c^{-1}.
\]
Taking maxima over the finitely many degrees allowed by \cref{prop:induced}, and finally multiplying the connection bound by $p$, gives $s(r,b,p)$ and $N(r,b,p)$.

Thus the radius and connection bound in \cref{thm:A} are determined by the stated ramification data and the fixed exponential convergence input. The degrees of the auxiliary coefficient extensions and the absolute ramification index of $K$ do not enlarge the final radius or norm bound.

The same choices of $s(r,b,p)$ and $N(r,b,p)$ apply to \cref{thm:B}: extending the prescribed Frobenius does not shrink the model annulus or change the connection bound. These constants are independent of $K$ and of the Frobenius slopes.


\begin{thebibliography}{GKU22}
\setlength{\baselineskip}{11pt}
\bibitem[And02]{And02}
Yves Andr\'e, Filtrations de type Hasse--Arf et monodromie $p$-adique, \emph{Inventiones Mathematicae} \textbf{148} (2002), no.~2, 285--317. \href{https://doi.org/10.1007/s002220100207}{doi:10.1007/s002220100207}.

\bibitem[Ber90]{Ber90}
Vladimir G. Berkovich, \emph{Spectral Theory and Analytic Geometry over Non-Archimedean Fields}, Mathematical Surveys and Monographs, vol.~33, American Mathematical Society, Providence, RI, 1990. \href{https://doi.org/10.1090/surv/033}{doi:10.1090/surv/033}.

\bibitem[BGR84]{BGR84}
Siegfried Bosch, Ulrich G\"untzer, and Reinhold Remmert, \emph{Non-Archimedean Analysis: A Systematic Approach to Rigid Analytic Geometry}, Grundlehren der mathematischen Wissenschaften, vol.~261, Springer-Verlag, Berlin, 1984.

\bibitem[Bor02]{Bor02}
James Borger, Kato's conductor and generic residual perfection, preprint, 2002. \href{https://arxiv.org/abs/math/0112306v2}{arXiv:math/\allowbreak 0112306v2} (Proposition~7.1 and proof).

\bibitem[Bry83]{Bry83}
Jean-Luc Brylinski, Th\'eorie du corps de classes de Kato et rev\^etements ab\'eliens de surfaces, \emph{Annales de l'Institut Fourier} \textbf{33} (1983), no.~3, 23--38. \href{https://doi.org/10.5802/aif.929}{doi:10.5802/aif.929}.

\bibitem[CP09]{CP09}
Bruno Chiarellotto and Andrea Pulita, Arithmetic and differential Swan conductors of rank one representations with finite local monodromy, \emph{American Journal of Mathematics} \textbf{131} (2009), no.~6, 1743--1794. \href{https://arxiv.org/abs/0711.0701}{arXiv:0711.0701}.

\bibitem[Con]{ConSZ}
Keith Conrad, \emph{The Schur--Zassenhaus theorem}, expository notes. \href{https://kconrad.math.uconn.edu/blurbs/grouptheory/schurzass.pdf}{Author's notes}.

\bibitem[Cre98]{Cre98}
Richard Crew, Finiteness theorems for the cohomology of an overconvergent isocrystal on a curve, \emph{Annales scientifiques de l'\'Ecole Normale Sup\'erieure} (4) \textbf{31} (1998), no.~6, 717--763. \href{https://doi.org/10.1016/S0012-9593(99)80001-9}{doi:10.1016/S0012-9593(99)80001-9}.

\bibitem[Eis95]{Eisenbud}
David Eisenbud, \emph{Commutative Algebra with a View Toward Algebraic Geometry}, Graduate Texts in Mathematics, vol.~150, Springer-Verlag, New York, 1995.

\bibitem[Gor68]{Gor68}
Daniel Gorenstein, \emph{Finite Groups}, Harper \& Row, New York, 1968.

\bibitem[GKU22]{GKU22}
Thomas Grubb, Kiran S. Kedlaya, and James Upton, A cut-by-curves criterion for overconvergence of $F$-isocrystals, preprint, 2022. \href{https://arxiv.org/abs/2202.03604}{arXiv:2202.03604}.

\bibitem[Isa76]{Isaacs}
I. Martin Isaacs, \emph{Character Theory of Finite Groups}, Academic Press, New York, 1976; reprinted by AMS Chelsea Publishing, 2006.

\bibitem[Ked04]{Ked04}
Kiran S. Kedlaya, A $p$-adic local monodromy theorem, \emph{Annals of Mathematics} (2) \textbf{160} (2004), no.~1, 93--184. \href{https://doi.org/10.4007/annals.2004.160.93}{doi:10.4007/annals.2004.160.93}.

\bibitem[Ked05]{Ked05}
Kiran S. Kedlaya, Local monodromy of $p$-adic differential equations: an overview, \emph{International Journal of Number Theory} \textbf{1} (2005), no.~1, 109--154. \href{https://arxiv.org/abs/math/0501361v2}{arXiv:math/0501361v2} (numbering cited here).

\bibitem[Ked22]{Ked22}
Kiran S. Kedlaya, \emph{$p$-adic Differential Equations}, second edition, Cambridge Studies in Advanced Mathematics, vol.~199, Cambridge University Press, Cambridge, 2022.

\bibitem[Mat95]{Mat95}
Shigeki Matsuda, Local indices of $p$-adic differential operators corresponding to Artin--Schreier--Witt coverings, \emph{Duke Mathematical Journal} \textbf{77} (1995), no.~3, 607--625. \href{https://doi.org/10.1215/S0012-7094-95-07719-9}{doi:10.1215/S0012-7094-95-07719-9}.

\bibitem[Mat02]{Mat02}
Shigeki Matsuda, Katz correspondence for quasi-unipotent overconvergent isocrystals, \emph{Compositio Mathematica} \textbf{134} (2002), no.~1, 1--34. \href{https://doi.org/10.1023/A:1020233201909}{doi:10.1023/A:1020233201909}.

\bibitem[Meb02]{Meb02}
Zoghman Mebkhout, Analogue $p$-adique du th\'eor\`eme de Turrittin et le th\'eor\`eme de la monodromie $p$-adique, \emph{Inventiones Mathematicae} \textbf{148} (2002), no.~2, 319--351. \href{https://doi.org/10.1007/s002220100208}{doi:10.1007/s002220100208}.

\bibitem[Pul07]{Pul07}
Andrea Pulita, Rank one solvable $p$-adic differential equations and finite Abelian characters via Lubin--Tate groups, \emph{Mathematische Annalen} \textbf{337} (2007), no.~3, 489--555. \href{https://doi.org/10.1007/s00208-006-0040-8}{doi:10.1007/s00208-006-0040-8}.

\bibitem[SOS89]{SOS89}
Tsutomu Sekiguchi, Frans Oort, and Noriyuki Suwa, On the deformation of Artin--Schreier to Kummer, \emph{Annales scientifiques de l'\'Ecole Normale Sup\'erieure} (4) \textbf{22} (1989), no.~3, 345--375. \href{https://doi.org/10.24033/asens.1586}{doi:10.24033/asens.1586}.

\bibitem[Ser77]{SerreRep}
Jean-Pierre Serre, \emph{Linear Representations of Finite Groups}, translated by Leonard L. Scott, Graduate Texts in Mathematics, vol.~42, Springer-Verlag, New York, 1977.

\bibitem[Ser79]{SerreLF}
Jean-Pierre Serre, \emph{Local Fields}, translated by Marvin Jay Greenberg, Graduate Texts in Mathematics, vol.~67, Springer-Verlag, New York, 1979.

\bibitem[ST68]{ST68}
Jean-Pierre Serre and John Tate, Good reduction of abelian varieties, \emph{Annals of Mathematics} (2) \textbf{88} (1968), no.~3, 492--517. \href{https://annals.math.princeton.edu/1968/88-3/p05}{Publisher's page}.

\bibitem[Stacks]{Stacks}
The Stacks Project Authors, \emph{The Stacks Project}, Lemma~41.14.3, \href{https://stacks.math.columbia.edu/tag/0EBS}{Tag~0EBS}.

\end{thebibliography}
\end{document}